\documentclass[12pt]{article}

\usepackage{amsmath,amssymb,amsbsy,amsfonts,amsthm,latexsym,
            amsopn,amstext,amsxtra,amscd,stmaryrd}

\usepackage[mathscr]{eucal}

\usepackage{color}
\usepackage[usenames,dvipsnames,svgnames,table]{xcolor}

\usepackage[title,titletoc,toc]{appendix}

\begin{document}

\newtheorem{theorem}{Theorem}[section]
\newtheorem{lemma}[theorem]{Lemma}
\newtheorem{corollary}[theorem]{Corollary}
\newtheorem{proposition}[theorem]{Proposition}

\theoremstyle{definition}
\newtheorem*{definition}{Definition}
\newtheorem*{remark}{Remark}
\newtheorem*{remarks}{Remarks}
\newtheorem*{example}{Example}

\numberwithin{equation}{section}

% CALLIGRAPHIC ALPHABET

\def\cA{\mathcal A}
\def\cB{\mathcal B}
\def\cC{\mathcal C}
\def\cD{\mathcal D}
\def\cE{\mathcal E}
\def\cF{\mathcal F}
\def\cG{\mathcal G}
\def\cH{\mathcal H}
\def\cI{\mathcal I}
\def\cJ{\mathcal J}
\def\cK{\mathcal K}
\def\cL{\mathcal L}
\def\cM{\mathcal M}
\def\cN{\mathcal N}
\def\cO{\mathcal O}
\def\cP{\mathcal P}
\def\cQ{\mathcal Q}
\def\cR{\mathcal R}
\def\cS{\mathcal S}
\def\cU{\mathcal U}
\def\cT{\mathcal T}
\def\cV{\mathcal V}
\def\cW{\mathcal W}
\def\cX{\mathcal X}
\def\cY{\mathcal Y}
\def\cZ{\mathcal Z}

% SCRIPT ALPHABET

\def\sA{\mathscr A}
\def\sB{\mathscr B}
\def\sC{\mathscr C}
\def\sD{\mathscr D}
\def\sE{\mathscr E}
\def\sF{\mathscr F}
\def\sG{\mathscr G}
\def\sH{\mathscr H}
\def\sI{\mathscr I}
\def\sJ{\mathscr J}
\def\sK{\mathscr K}
\def\sL{\mathscr L}
\def\sM{\mathscr M}
\def\sN{\mathscr N}
\def\sO{\mathscr O}
\def\sP{\mathscr P}
\def\sQ{\mathscr Q}
\def\sR{\mathscr R}
\def\sS{\mathscr S}
\def\sU{\mathscr U}
\def\sT{\mathscr T}
\def\sV{\mathscr V}
\def\sW{\mathscr W}
\def\sX{\mathscr X}
\def\sY{\mathscr Y}
\def\sZ{\mathscr Z}

% FRAKTUR ALPHABET

\def\fA{\mathfrak A}
\def\fB{\mathfrak B}
\def\fC{\mathfrak C}
\def\fD{\mathfrak D}
\def\fE{\mathfrak E}
\def\fF{\mathfrak F}
\def\fG{\mathfrak G}
\def\fH{\mathfrak H}
\def\fI{\mathfrak I}
\def\fJ{\mathfrak J}
\def\fK{\mathfrak K}
\def\fL{\mathfrak L}
\def\fM{\mathfrak M}
\def\fN{\mathfrak N}
\def\fO{\mathfrak O}
\def\fP{\mathfrak P}
\def\fQ{\mathfrak Q}
\def\fR{\mathfrak R}
\def\fS{\mathfrak S}
\def\fU{\mathfrak U}
\def\fT{\mathfrak T}
\def\fV{\mathfrak V}
\def\fW{\mathfrak W}
\def\fX{\mathfrak X}
\def\fY{\mathfrak Y}
\def\fZ{\mathfrak Z}

% BLACKBOARD BOLD

\def\C{{\mathbb C}}
\def\F{{\mathbb F}}
\def\K{{\mathbb K}}
\def\L{{\mathbb L}}
\def\N{{\mathbb N}}
\def\Q{{\mathbb Q}}
\def\R{{\mathbb R}}
\def\Z{{\mathbb Z}}

% SOME STANDARD DEFINITIONS

\def\eps{\varepsilon}
\def\mand{\qquad\mbox{and}\qquad}
\def\\{\cr}
\def\({\left(}
\def\){\right)}
\def\[{\left[}
\def\]{\right]}
\def\<{\langle}
\def\>{\rangle}
\def\fl#1{\left\lfloor#1\right\rfloor}
\def\rf#1{\left\lceil#1\right\rceil}
\def\le{\leqslant}
\def\ge{\geqslant}
\def\ds{\displaystyle}

\def\xxx{\vskip5pt\hrule\vskip5pt}
\def\yyy{\vskip5pt\hrule\vskip2pt\hrule\vskip5pt}
\def\imhere{ \xxx\centerline{\sc I'm here}\xxx }

\newcommand{\comm}[1]{\marginpar{
\vskip-\baselineskip \raggedright\footnotesize
\itshape\hrule\smallskip\tiny\textcolor{Blue}{#1}\par\smallskip\hrule}}

% SPECIAL DEFINITIONS FOR THIS PAPER

\def\distr{{\tt d}}
\def\b{{\frak b}}
\def\c{{\frak c}}
\def\error{{\underline o(1)}}
\def\e{\mathbf{e}}
\def\cc#1{\textcolor{Blue}{#1}}

%%%%%%%%%%%%%%%%%%%%%%%%%%%%%%%%%%%%%%%%%
%%%%%%%%%%  PAPER STARTS HERE  %%%%%%%%%%
%%%%%%%%%%%%%%%%%%%%%%%%%%%%%%%%%%%%%%%%%

\title{\sc Convolutions with probability distributions,
zeros of $L$-functions, and the least quadratic nonresidue}

\author{
{\sc William D.\ Banks} \\
{Department of Mathematics} \\
{University of Missouri} \\
{Columbia, MO 65211 USA} \\
{\tt bankswd@missouri.edu}
\and
{\sc Konstantin A.\ Makarov} \\ 
{Department of Mathematics} \\
{University of Missouri} \\
{Columbia, MO 65211 USA} \\
{\tt makarovk@missouri.edu}}

\maketitle

\newpage

\begin{abstract}
Let $\distr$ be the density of a probability distribution
that is compactly supported in the positive semi-axis. Under certain mild conditions
we show that 
$$
\lim_{x\to\infty}x\sum_{n=1}^\infty
\frac{\distr^{*n}(x)}{n}=1,\qquad\text{where}\quad
\distr^{*n}:=\underbrace{\,\distr *\distr*\cdots*\distr\,}_{n\text{~times}}.
$$
We also  show that if $c>0$ is a
given constant for which
the function $f(k):=\widehat\distr(k)-1$ does not vanish on the line
$\{k\in\C:\Im\,k=-c\}$, where $\widehat\distr$ is the Fourier transform of $\distr$,
then one has the asymptotic expansion
$$
\sum_{n=1}^\infty\frac{\distr^{*n}(x)}{n}=\frac{1}{x}\bigg(1+\sum_k m(k)
e^{-ikx}+O(e^{-c x})\bigg)\qquad (x\to +\infty),
$$
where the sum is taken over those zeros $k$ of $f$ that lie in
the strip $\{k\in\C:-c<\Im\,k<0\}$, $m(k)$ is the multiplicity of any such zero, and 
the implied constant depends only on $c$.
For a given distribution 
of this type, we briefly describe the location
of the zeros $k$ of $f$ in the lower half-plane $\{k\in\C:\Im\,k<0\}$.

For an odd prime $p$, let $n_0(p)$ be the least natural number such that $(n|p)=-1$,
where $(\cdot|p)$ is the Legendre symbol.
As an application of our work on probability distributions, 
we generalize a well known result
of Heath-Brown concerning the exhibited behavior of the Dirichlet $L$-function
$L(s,(\cdot|p))$ under the assumption that the Burgess bound
$n_0(p)\ll p^{1/(4\sqrt{e})+\eps}$ cannot be improved.
\end{abstract}

\newpage

\tableofcontents

\newpage

\section{Statement of results}

In this paper, we establish a very general theorem concerning
convolutions of certain compactly supported probability distributions.
As an application to analytic number theory, we use our theorem
to generalize a well known result of Heath-Brown concerning the behavior
of the Dirichlet $L$-function attached to the Legendre symbol under an
assumption that the Burgess bound on the least quadratic nonresidue
cannot be improved.

\subsection{Convolutions with probability distributions}
\label{sec:c-intro}

Let $\distr$ be the density of a probability distribution that is
supported in a finite interval $[a,b]$ with $a>0$.
Assume that $\distr$ is twice continuously differentiable on $(a,b)$, and that
$\distr(a)\distr(b)\ne 0$. Put
\begin{equation}
\label{eq:series}
F_\distr(x):=\sum_{n=1}^\infty\frac{\distr^{*n}(x)}{n}\qquad(x>0),
\end{equation}
where $\distr^{*n}$ denotes the $n$-fold convolution of $\distr$ with itself, i.e.,
$$
\distr^{*n}:=\underbrace{\,\distr*\distr*\cdots*\distr\,}_{n\text{~times}}\qquad(n\in\N).
$$
Since $\distr^{*n}(x)=0$ whenever $x<na$, for every $x>0$ the
series~\eqref{eq:series} has only finitely many nonzero terms, hence
the function $F_\distr$ is well-defined pointwise (but not absolutely summable; see Corollary \ref{cor:notsum}).

In this paper, we show that the leading term
in the asymptotic expansion of $F_\distr(x)$ as
$x\to\infty$ is universal, i.e., it does not depend on the particular choice of 
$\distr$  for a wide class of  distributions, 
while the (exponentially small) higher order terms of the asymptotics
are determined by the roots in the lower half-plane of the equation 
$$
\widehat\distr(k)=1.
$$
Here  $\widehat\distr$ denotes the Fourier transform of $\distr$, which is an entire function given by
$$
\widehat\distr(k):=\int_a^b\distr(x)e^{ikx}\,dx\qquad(k\in\C).
$$

\newpage

\begin{theorem}
\label{thm:c-main}
For any constant $c>0$, the function $f(k):=\widehat\distr(k)-1$
has only finitely many zeros in the strip
$\Pi_c:=\{k\in\C:-c<\Im\,k<0\}$.
Suppose that $f$ does not vanish on the line $\{k\in\C:\Im\,k=-c\}$. Then
\begin{equation}\label{asympt}
F_\distr(x)=\frac{1}{x}\bigg(1+\sum_k m(k)
e^{-ikx}+E(c, x)e^{-c x}\bigg)\qquad(x>0),
\end{equation}
where the sum is taken over those zeros $k$ of $f$ that lie in
$\Pi_c$, $m(k)$ is the multiplicity of any such zero, and
\begin{equation}\label{ecx}
E(c, x):=
\frac{1}{2\pi i}\int_\R \left (
\frac{\widehat\distr'(u-ic)}{1-\widehat\distr(u-ic)}
-\widehat\distr'(u-ic)\right )\,e^{-iux}\,du.
\end{equation}
\end{theorem}

\medskip

\begin{remarks}
Note that for any $c>0$ satisfying the hypotheses of Theorem~\ref{thm:c-main},
the quantity $E(c,x)$ satisfies the uniform bound
$$
|E(c,x)|\le\frac{1}{2\pi}\,\big\|f_c\big\|_{L^1(\R)}\qquad(x>0),
$$
where
$$
f_c(u):=\frac{\widehat\distr'(u-ic)}{1-\widehat\distr(u-ic)}
-\widehat\distr'(u-ic)\qquad(u\in\R).
$$
\end{remarks}

\medskip

The proof of Theorem~\ref{thm:c-main} is given in
\S\ref{sec:proofc}, and in \S\ref{sec:solns} we briefly explore the location of the zeros
of the function $f(k):=\widehat\distr(k)-1$ that lie in the lower half-plane. 

\subsection{The least quadratic nonresidue}

For any odd prime $p$, let $n_0(p)$ denote the least positive quadratic
nonresidue modulo $p$; that is,
$$
n_0(p):=\min\{n\in\N:(n|p)=-1\},
$$
where $(\cdot|p)$ is the \emph{Legendre symbol}.
The first nontrivial bound on $n_0(p)$ was given by
Gauss~\cite[Article~129]{Gauss}, who showed that
$n_0(p)<2\sqrt{p}+1$ holds for every prime $p\equiv 1\pmod 8$.
Vinogradov~\cite{Vino} proved that $n_0(p)\ll p^\kappa$ holds for all
primes $p$ provided that $\kappa>1/(2\sqrt{e})$, and later,
Burgess~\cite{Burg} extended this range to include all real numbers
$\kappa>1/(4\sqrt{e})$.  The latter result has not been
improved since 1957.

An old conjecture of Vinogradov asserts that the bound $n_0(p)\ll p^\eps$ holds for
every fixed $\eps>0$. Linnik~\cite{Linnik0} showed that Vinogradov's conjecture
is true under the \emph{Extended Riemann Hypothesis} (ERH). A decade later,
Ankeny~\cite{Ankeny} proved that the stronger bound
$n_0(p)\ll (\log p)^2$ holds under the ERH.

It is natural to wonder what bounds on $n_0(p)$ can be established under
weaker conditional hypotheses than the ERH.
The pioneering work in this direction (which largely motivates the present paper)
is an unpublished analysis of Heath-Brown concerning
the behavior of the Dirichlet $L$-function
$L(s,(\cdot|p))$ under an assumption that the Burgess bound is
\emph{tight}, i.e., that the lower bound $n_0(p)\ge p^{1/(4\sqrt{e})}$
holds for infinitely many primes $p$; we refer the reader to
Diamond \emph{et al}~\cite[Appendix]{DMV} for a superb
account of Heath-Brown's methods and results.

In this paper, we modify and extend Heath-Brown's ideas as follows.
Throughout, let $\kappa,\lambda$ be fixed real numbers such that
\begin{equation}\label{lamkaphyp}0<\kappa\le \frac{\lambda}{\sqrt{e}}\le \frac{1}{4\sqrt{e}}.
\end{equation} 
For every odd prime $p$, put
\begin{equation}
\label{eq:NpXdefi}
\sN_p(X):=\{n\le X:(n|p)=-1\}\qquad(X>0).
\end{equation}
We assume that there is an infinite set of primes $\sP$ for which
\begin{equation}
\label{eq:n0pkap}
n_0(p)\ge p^\kappa\qquad(p\in\sP).
\end{equation}
 Our aim is to understand how the zeros of
$L(s,(\cdot|p))$ are constrained by the condition \eqref{eq:n0pkap} 
(as previously mentioned, such a set $\sP$ cannot exist under the ERH by the
work of Linnik~\cite{Linnik0}).

In addition to \eqref{eq:n0pkap} we also assume that for any fixed $\theta\ge 0$ the
estimate
\begin{equation}
\label{eq:Np-est}
\big|\sN_p(p^\theta)\big|=(\delta(\theta)+\error)p^\theta
\end{equation}
holds, where $\error$ denotes an error term that tends to zero
as $p\to\infty$ with primes $p$ lying in the set $\sP$, and
$\delta$ is a function of the form
\begin{equation}
\label{eq:intrho}
\delta(\theta):=\frac12\int_0^\theta\distr(u)\,du\qquad(\theta\ge 0)
\end{equation}
with some probability distribution $\distr$ that is supported in the interval $[\kappa,\lambda]$ and twice continuously
differentiable on $(\kappa,\lambda)$,
with $\distr(\kappa)\distr(\lambda)\ne 0$.

We remark that  the inequalities in \eqref{lamkaphyp} are optimal in a
certain sense (see Lemmas \ref{lem:hildebrand}
and Theorem \ref{low} below). The main result of the paper is as follows.

\begin{theorem}
\label{thm:b-main}
Under the hypotheses
\eqref{eq:n0pkap},
\eqref{eq:Np-est}
and \eqref{eq:intrho}, for every
nonzero root $k$
of the equation $\widehat\distr(k)=1$ there is a
complex sequence $(\varrho_p)_{p\in\sP}$ with
$L(\varrho_p,(\cdot|p))=0$ such that $(\varrho_p-1)\log p\to -ik$ as $p\to\infty$
with $p\in\sP$.
\end{theorem}

\medskip

\begin{remarks}
In the special case that $\kappa:=1/(4\sqrt{e})$ and $\lambda:=1/4$,
we show in \S\ref{sec:heath-brown}
that under hypothesis  \eqref{eq:n0pkap}  the condition \eqref{eq:Np-est}
is automatically met with the function $\delta$ given by
\begin{equation}
\label{eq:dragon}
\delta(\theta):=\left\{
\begin{array}{ll}
    0 &\quad \text{if}\quad \hbox{$0\le \theta\le 1/(4\sqrt{e})$,} \\
    \log(4\theta\sqrt{e}) &\quad  \text{if}\quad1/(4\sqrt{e})\le \theta\le 1/4,\\
    1/2 &\quad  \text{if}\quad\hbox{$\theta\ge 1/4$\, ,} \\
\end{array}
\right.
\end{equation}
and the probability distribution $\distr$ defined by
\begin{equation}
\label{eq:Burgessrho}
\distr(x):=\left\{
\begin{array}{ll}
    2x^{-1} &\quad  \text{if}\quad \hbox{$1/(4\sqrt{e})\le x\le 1/4$,}\\
    0 &\quad \hbox{otherwise.} \\
\end{array}
\right.
\end{equation}
Then, from the conclusion of Theorem~\ref{thm:b-main}
we recover the aforementioned result of Heath-Brown.
We also note that in any application
of Theorem~\ref{thm:b-main} it is useful to have information about
the location of the zeros of the function $f(k):=\widehat\distr(k)-1$.
General results of this nature are given in Proposition~\ref{prop:zeros},
where we outline a standard method for obtaining
such information.

We also remark that the normalization factor $\frac12$ in
hypothesis~\eqref{eq:intrho} is chosen to meet the unconditional requirement that 
$$
\lim_{p\to \infty}\frac{\big|\sN_p(p^\theta)\big|}{p^\theta}=\frac12 \qquad (\theta \ge 1/4);
$$
see Lemma \ref{lem:hildebrand}.
\end{remarks}

\medskip

The proof of Theorem~\ref{thm:b-main} (see \S\ref{sec:proofb} below)
can be summarized as follows. First, we show that the limit
\begin{equation}
\label{eq:S1defn}
S_1(\theta):=\lim_{\substack{p\to\infty\\ p\in\sP}}
\sum_{\substack{q\le p^\theta\\ (q|p)=-1}}q^{-1}
\end{equation}
exists for all $\theta\ge 0$, where the sum is taken
over prime nonresidues $q\le p^\theta$.  Using properties of the Laplace
transform we show that $S_1$ is continuously differentiable on $(\lambda,\infty)$
and that \begin{equation}
\label{eq:laplace_rocks}
S'_1(\theta)=\frac{1}{2}\sum_{n\in\N}\frac{\distr^{*n}(\theta)}{n}\qquad(\theta>\lambda),
\end{equation}
where $\distr^{*n}$ denotes the $n$-fold convolution
$\distr*\cdots*\distr$ as before.  Taking into account
Theorem~\ref{thm:c-main}, for any fixed $c>0$ we obtain
an estimate of the form
$$
S'_1(\theta)=\frac{1}{2\theta}+  \frac{1}{2\theta}\sum_k m(k)e^{-ik\theta}
+\frac{E(c,\theta)e^{-c\,\theta}}{2 \theta},
$$
where $E(c,\theta)$ is given by \eqref{ecx}.
On the other hand, expressing the derivative $S'_1(\theta)$ as a limit of
difference quotients and using standard estimates from number theory,
we derive that for any fixed $c>0$
one has
$$ 
S_1'(\theta)=\frac{1}{2\theta}+\frac{1}{2\theta}
\lim_{\substack{p\to\infty\\ p\in\sP}}
\sum_\varrho\frac{\widetilde m(\varrho)}{\varrho p^{(1-\varrho)\theta}}
+O(e^{-c\,\theta}),
$$
where each sum runs over the distinct zeros $\varrho=\beta+i\gamma$ of $L(s,(\cdot|p))$
in the region determined by the inequalities
$$
\beta>1-c/\log p \mand |\gamma|\le p,
$$  
and $\widetilde m(\varrho)$ is the multiplicity of any such zero.
A comparison of these two relations leads to the statement of Theorem~\ref{thm:b-main}.

Not too surprisingly,
 our proof of Theorem~\ref{thm:b-main} incorporates
principles that figure prominently in treatments of Linnik's Theorem,
including the log-free zero-density estimate (see Linnik~\cite{Linnik1}) and
the Deuring-Heilbronn phenomenon (see Linnik~\cite{Linnik2}).  On the other
hand, our method of applying the Laplace transform to derive
\eqref{eq:laplace_rocks} appears to be new.

\section{Proof of Theorem~\ref{thm:c-main}}
\label{sec:proofc}

We continue to assume that $\distr$ has the properties listed
in \S\ref{sec:c-intro}; that is, the function $\distr$ is twice
continuously differentiable on $(a,b)$, and
$\distr(a)\distr(b)\ne 0$.

\begin{lemma}
\label{lem:l1}
We have
$$
\widehat\distr(k)=1+id_1k-\tfrac12d_2k^2+O(k^3)\qquad(k\to 0),
$$
where
$$
d_1:=\int_a^b x\,\distr(x)\,dx\mand d_2:=\int_a^b x^2\,\distr(x)\,dx.
$$
Also,
$$
\widehat\distr(k)=\frac{1}{ik}\big(\distr(b)e^{ikb}
-\distr(a)e^{ika}\big)+O(k^{-2})\qquad(k\to\infty). 
$$
\end{lemma}

\begin{proof} The first representation follows by expanding $e^{ikx}$ as
a power series around $k=0$, whereas the second is obtained using integration
by parts
\begin{align*}
\widehat\distr(k)=\int_a^b\distr(x)e^{ikx}\,dx 
=&\frac{1}{ik}\big(\distr(b)e^{ikb}-\distr(a)e^{ika}\big)
+\frac{1}{k^2}\big(\distr'(b)e^{ikb}
-\distr'(a)e^{ika}\big)\\
&-\frac{1}{k^2}\int_a^b\distr''(x)e^{ikx}\,dx
\end{align*}
together with the fact that
$\distr''(x)$ is a continuous function on  [a,b].
\end{proof}

\medskip

\begin{remark}\label{remrem1}
In the lower half-plane we have the estimate
$$
\widehat\distr(k)=O\left (\frac{e^{|\Im\,k|b}}{|k|}\right),
$$ 
which holds uniformly with respect to $\arg k$.
\end{remark}

\medskip

\begin{lemma}
\label{lem:l2}
$F_\distr\in L^2(\R)$.
\end{lemma}

\begin{proof} 
It is enough to prove that the series \eqref{eq:series} converges in $L^2(\R)$.
In turn, since $\widehat\distr^n$ is the Fourier transform of
$\distr^{*n}$ for each $n\in\N$, it suffices to show that the series
\begin{equation}
\label{eq:series2}
\sum_{n=1}^\infty \frac{\widehat\distr(k)^n}{n}\qquad(k\in\R)
\end{equation}
converges in $L^2(\R)$.

\bigskip\textit{Step 1}.
First we note that
\begin{equation}
\label{eq:coffee}
|\widehat\distr(k)|=\bigg|\int_a^b\distr(x)e^{ikx}\,dx\bigg|
<\int_a^b\distr(x)\,dx=1\qquad(k\in\R\setminus\{0\})
\end{equation}
since $\distr$ is nonnegative and not identically zero, hence the series
\eqref{eq:series2} converges uniformly on every compact set
$\Omega\subset\R\setminus\{0\}$; this proves, in particular, that the series
\eqref{eq:series2} converges in $L^2(\Omega)$. 

\bigskip\textit{Step 2}.
By Lemma~\ref{lem:l1} it is easy to see that there exists $\delta>0$ such that
$$
|\widehat\distr(k)|\le\sqrt{1- C_1k^2}\qquad(k\in[-\delta,\delta])
$$
holds for some positive constant $C_1$ that is less than
$$
d_2-d_1^2=\int_a^b x^2\,\distr(x)\,dx -\bigg(\int_a^b x\,\distr(x)\,dx\bigg)^2>0.
$$
Using  Laplace's method (see, e.g., \cite[Ch. 3, Sec. 7]{Ol}), we find that
$$
\int_{-\delta}^\delta|\widehat\distr(k)|^n\,dk
\le\int_{-\delta}^\delta (1-C_1k^2)^{n/2}\,dk
=\left (\frac{2\pi}{C_1n}\right )^{1/2}+O(n^{-3/2});
$$
that is, for $\delta$ chosen as above, the inequality
$$
\int_{-\delta}^\delta |\widehat\distr(k)|^n\,dk\le \frac{C_2}{\sqrt{n}}\qquad(n\in\N)
$$
holds for some constant $C_2>0$. Hence,
\begin{equation}
\label{eq:est}
\|\widehat\distr^n\|_{L^2(-\delta, \delta)}\le  \frac{\sqrt{C_2}}{(2n)^{1/4}}.
\end{equation}
Now, for any natural numbers $M>N$, from \eqref{eq:est} we deduce that
\begin{align*}
\bigg\|\sum_{n=N}^M \frac{\widehat\distr^n}{n}\bigg\|_{L^2(-\delta,\delta)}
&\le \sum_{n=N}^M \frac1n \|\widehat\distr^n\|_{L^2(-\delta, \delta)}\le \frac{\sqrt{C_2}}{2^{1/4}}\sum_{n=N}^M\frac{1}{n^{5/4}},
\end{align*}
which shows that the series \eqref{eq:series2} converges in $L^2(-\delta,\delta)$.

\bigskip\textit{Step 3}.
By Lemma \ref{lem:l1} it is also clear 
that  there is a constant $A$ such that for every sufficiently
large $R>0$ the inequality
$$
|\widehat\distr(k)|\le \frac{A}{|k|}\qquad (|k|>R).
$$
holds. Increasing $R$ if necessary, we can assume
that $R>A$; then, for any natural numbers $M>N\ge 2$ we have
\begin{align*}
\bigg\|\sum_{n=N}^M \frac{\widehat\distr^n}{n}\bigg\|_{L^2(\R\setminus[-R,R])}
&\le\sum_{n=N}^M\frac{1}{n}\left (A^{2n}\int_{|k|>R}\frac{dk}{|k|^{2n}}\right )^{1/2}\\
&\le \sqrt{2R}\sum_{n=N}^M\frac{1}{n\sqrt{2n-1}} \left (\frac{A}{R}\right )^n
\\&\le \sqrt{2R}\sum_{n=N}^M\frac{1}{n\sqrt{2n-1}},
\end{align*}
which shows that the series \eqref{eq:series2} converges 
in $L^2(\R\setminus[-R,R])$.

\bigskip
 
Combining the results of the steps above, we conclude that the series
\eqref{eq:series2} converges in $L^2(\R)$ as required, and the lemma is proved.
\end{proof}

\begin{lemma}\label{lem:2.3} We have
\begin{equation} 
\label{eq:both}
F_\distr(x)-\distr(x)=-\frac{1}{2\pi}\int_\R\left (\log(1-\widehat\distr(k))+ \widehat\distr(k)\right )e^{-ikx}\,dk\qquad(x\in \R).
\end{equation}
\end{lemma}

\begin{proof}
From the bound \eqref{eq:coffee} we see that
$$
\widehat F_\distr(k)=\sum_{n=1}^\infty \frac{\widehat\distr(k)^n}{n}
=-\log(1-\widehat\distr(k))\qquad (k\in\R\setminus\{0\}).
$$ 
Since  $F_\distr\in L^2(\R)$ by Lemma \ref{lem:l2}, one obtains that
$$
F_\distr(x)=-\frac{1}{2\pi}\,\text{\rm l.i.m.} \int_\R
\log(1-\widehat\distr(k))e^{-ikx}~dk\qquad(\text{\rm a.e.~}x\in\R).
$$
Therefore,
\begin{equation}
\label{eq:bott}
F_\distr(x) -\distr(x)=-\frac{1}{2\pi}\,\text{\rm l.i.m.} \int_\R
\left (\log(1-\widehat\distr(k)) + \widehat\distr(k)\right )e^{-ikx}~dk
\end{equation}
for almost all $x\in \R$.

However, since the integrand
$\log(1-\widehat\distr\,)+\widehat\distr$
is absolutely summable by Lemma~\ref{lem:l1},
we see that \eqref{eq:bott} implies \eqref{eq:both}
since the functions on either side of \eqref{eq:both} are continuous, and 
every $L^2$-function (i.e., equivalence class of functions) has at most one
continuous representative.
\end{proof}

\medskip

\begin{proof}[Proof of Theorem~\ref{thm:c-main}]
Since $\distr(x)=0$ for $x>b$, from Lemma \ref{lem:2.3} it follows that 
\begin{align*}
xF_\distr(x)=&-\frac{x}{2\pi}\int_\R\left (\log(1-\widehat\distr(k))+\widehat\distr(k)\right )e^{-ikx}\,dk \qquad (x>b).
\end{align*}
Recall that $\log (1-\widehat \distr)+\widehat \distr\in L^1(\R)$, and therefore
\begin{align*}
\int_\R\left (\log(1-\widehat\distr(k))+\widehat\distr(k)\right )e^{-ikx}\,dk&
=\lim_{\varepsilon \to 0^+}\int_{\R\setminus(-\varepsilon, \varepsilon)}\left (\log(1-\widehat\distr(k))+\widehat\distr(k)\right )e^{-ikx}\,dk\\
&=\text{V.P.}\int_{\R}\left (\log(1-\widehat\distr(k))+\widehat\distr(k)\right )e^{-ikx}\,dk.
\end{align*}
Integration by parts yields the relation
\begin{equation}
\label{eq:chair}
xF_\distr(x)=\frac{1}{2}+\frac{1}{2\pi i}\,\text{V.P.}
\int_\R\left (\frac{\widehat\distr'(k)}{1-\widehat\distr(k)}-\widehat\distr'(k)\right )\,e^{-ikx}\,dk
\qquad(x>b)
\end{equation}
as Lemma~\ref{lem:l1} implies that
\begin{align*} 
&\lim_{\eps\to 0^+}
\(\[\log(1-\widehat\distr(-\eps)) +\widehat\distr(-\eps)\]e^{i\eps x}-
\[\log(1-\widehat\distr(\eps))+\widehat\distr(\eps)\]e^{-i\eps x}\)\\
&=\lim_{\eps\to 0^+}
\(\log(1-\widehat\distr(-\eps)) e^{i\eps x}-
\log(1-\widehat\distr(\eps))e^{-i\eps x}\)\\
&=\lim_{\eps\to 0^+}\left (\log(id_1\varepsilon)-\log(-id_1\varepsilon)\right )
=\pi i,
\end{align*}
with
$$
d_1=\int_a^b x\,\distr(x)\,dx >0,
$$
the mean of the distribution $\distr$.

Next, concerning the V.P.\ integral in \eqref{eq:chair} we have
\begin{align*}
\text{V.P.}\int_\R\left (\frac{\widehat\distr'(k)}{1-\widehat\distr(k)}-\widehat\distr'(k)\right )\,e^{-ikx}\,dk
:&=\lim_{\eps\to 0^+}\int_{\R\setminus(-\eps,\eps)}
\left (\frac{\widehat\distr'(k)}{1-\widehat\distr(k)}-\widehat\distr'(k) \right )\,e^{-ikx}\,dk\\
&=\lim_{\eps\to 0^+}(\cI_\eps-\cI'_\eps),
\end{align*}
where
\begin{align*}
\cI'_\eps&:=\int_{\Gamma'_\eps}
\left (\frac{\widehat\distr'(z)}{1-\widehat\distr(z)}-\widehat\distr'(z)\right )\,e^{-izx}\,dz\qquad\text{with}\quad
\Gamma'_\eps:=\{z\in\C:|z|=\eps,~\Im\,z<0\},\\
\cI_\eps&:=\int_{\Gamma_\eps}
\left (\frac{\widehat\distr'(z)}{1-\widehat\distr(z)}-\widehat\distr'(z)\right )\,e^{-izx}\,dz\qquad\text{with}\quad
\Gamma_\eps:=(-\infty,-\eps)\cup \Gamma'_\eps\cup (\eps, \infty).
\end{align*}
Here, $\Gamma'_\eps$ and $\Gamma_\eps$ are oriented so that $\Re\,z$
is increasing on each contour.  Using Lemma~\ref{lem:l1} again, it
is easy to see that
$$
\frac{1}{2\pi i}\lim_{\eps\to 0^+}\cI'_\eps=\frac12\,
\text{\rm Res}\bigg|_{z=0}\left (\frac{\widehat\distr'(z)}{1-\widehat\distr(z)}-\widehat\distr'(z)\right ) e^{-izx}=-\frac12.
$$
Using this information in \eqref{eq:chair} and applying the Residue Theorem,
we have
\begin{align*}
xF_\distr(x)&=\frac12+\frac12+\frac{1}{2\pi i}\lim_{\eps\to 0^+}\cI_\eps\\
&=1+\frac{1}{2\pi i}\int_{\Im\,k=-c}
\left (\frac{\widehat\distr'(k)}{1-\widehat\distr(k)}-
\widehat \distr'(k)\right )e^{-ikx}\,dk\\
&\qquad -\sum_k\text{\rm Res}\bigg|_{z=k}
\bigg(\frac{\widehat\distr'(z) }{1-\widehat\distr(z)}
-\widehat\distr'(z)\bigg )e^{-izx}+E_1+E_2,
\end{align*}
where the integral over $\{k\in\C:\Im\,k=-c\}$ is oriented
with $\Re\,k$ increasing,
the sum is taken over all roots $k$ of the equation
$\widehat\distr(k)=1$ for which $k\in\Pi_c$, 
and
\begin{align*}
E_1&:=-\frac{1}{2\pi}\lim_{R\to \infty}
\int_0^c\left (\frac{\widehat\distr'(-R-iu)}{1-\widehat\distr(-R-iu)} -\widehat \distr'(-R-iu)\right )\,e^{-ux+iR}\,du,\\
E_2&:=\frac{1}{2\pi}\lim_{R\to \infty}
\int_{-c}^0\left (\frac{\widehat\distr'(R+iu)}{1-\widehat\distr(R+iu)}
-\widehat \distr'(R-iu)\right )\,e^{ux-iR}\,du.
\end{align*}
Denoting by $m(k)$ the multiplicity of each root $k$ in the sum, 
we have
$$
\text{\rm Res}\bigg|_{z=k}\bigg(\frac{\widehat\distr'(z)}{1-\widehat\distr(z)}-\widehat \distr'(z)
\bigg)e^{-izx}
=-m(k)e^{-ikx}.
$$
Taking into account that $E_1=E_2=0$ by Remark~\ref{remrem1}, we
finish the proof.
\end{proof}

\begin{corollary}
\label{cor:notsum}
$F_\distr\in \left ( L^1_{\text{\rm w}}(\R_+)\setminus  L^1(\R_+)\right )\cap L^2(\R_+).$
\end{corollary}

\begin{proof}
The fact that $F_\distr\notin L^1(\R_+)$ follows from the definition \eqref{eq:series} and the observations that 
$$
\int_0^\infty\distr^{*n}(x)dx=1
$$
and that the series $\sum_{n=1}^\infty\frac1n$ diverges. 

The membership $F_\distr\in L^2(\R_+)$ is the content of
Lemma~\ref{lem:l2}.

Finally, $F_\distr$ belongs to the weak space $L^1_{\text{w}}(\R_+)$ since
$F_\distr$ is a bounded function and it admits the estimate
$$
F_\distr(x)=O(x^{-1})\qquad (x\to\infty)
$$
(this follows from the asymptotics \eqref{asympt}), hence
$$
\sup_{t>0}\, t \cdot \text{mes}\{x>0:|F(x)|>t\}<\infty.
$$
This completes the proof.
\end{proof}

\section{On solutions to the equation $\widehat\distr(k)=1$}
\label{sec:solns}

In this section we briefly describe the location of zeros of the function
$\widehat\distr(k)-1$ that lie in the lower half-plane.  Our results here concerning the distribution of the zeros are
not used in the proof of Theorem~\ref{thm:b-main} in \S\ref{sec:proofb}
below. However, as Theorem~\ref{thm:b-main} shows, these zeros suitably translated and rescaled  are zeros of the $L$-function.

In view of the remark following Lemma~\ref{lem:l1} we see that the
aforementioned zeros lie asymptotically close to solutions of the equation
\begin{equation}
\label{eq:basic}
e^{ikb}=\frac{ik}{\distr(b)}\qquad(\Im\,k<0).
\end{equation}
The solutions to \eqref{eq:basic} can be determined explicitly in terms
of the Lambert \text{$W$-function} or estimated using
standard methods going back to Horn \cite{Horn1,Horn2}
(see also Hardy \cite{Hardy},  Zdanovich \cite{Zd},
Pavlov~\cite{PP73,P73}, and Zworski~\cite{Zw87}).

\begin{proposition}
\label{prop:zeros}
{\rm (cf.\ \cite[Lemma 2]{DMV})}
The zeros of the equation $\widehat\distr(k)=1$ satisfy the asymptotic formula
$$
k_{\pm n}=\pm\frac{\pi}{b}(2n+\tfrac12)-\frac{i}{b}\log \frac{2\pi n}{b\,\distr(b)}
+o(1)\qquad(n\in\N,~n\to\infty).
$$
\end{proposition}

A heuristic argument proceeds as follows.
To find solutions to \eqref{eq:basic} we introduce a new variable 
$z=ikb$ and rewrite \eqref{eq:basic} in the form
$$
e^z=\alpha z\qquad\text{with}\quad\alpha:=(b\,\distr(b))^{-1}.
$$
We prepare this equation for ``bootstrapping'' by writing it in the form
$$
z=\log(\alpha z)+2\pi in
$$
with a fixed $n\in\N$.  We apply the Banach fixed point theorem,
starting the iterative process with
\begin{align*}
z^{(0)}&:=2 \pi i n,\\
z^{(1)}&:=\log(\alpha z^{(0)})+2\pi in=\log(2\pi n\alpha)+\pi i(2n+\tfrac12),
\end{align*}
and continuing in this way by putting
$$
z^{(j+1)}_n:=\log(\alpha z^{(j)})+2\pi in\qquad(j\ge 2).
$$
If $n$ and $j$ are large we see that 
$$
z^{(j)}_n=\log(2\pi n\alpha)+\pi i(2n+\tfrac12)+(\text{lower order terms}).
$$
Returning to the original variable $k$ we conclude that the zeros
$k_n$ with $\Re\,k_n>0$ and $\Im\,k_n<0$ satisfy
$$
k_n\sim\frac{\pi}{b}(2n+\tfrac{1}{2})-\frac{i}{b}\log(2\pi n\alpha)+o(1)
\qquad(n\to\infty).
$$
The heuristic argument is completed by noting that the zeros of
$\widehat\distr(\zeta)-1$ are located symmetrically with respect to the imaginary axis.

\section{Proof of Theorem~\ref{thm:b-main}}
\label{sec:proofb}

\subsection{Some technical lemmas}

For the proof of Theorem~\ref{thm:b-main} we need several technical results.

For any Dirichlet character $\chi$ we denote by
$N(\sigma,T,\chi)$  the number of zeros of $L(s,\chi)$ in
the region $\{s\in\C:\sigma\le\Re\,s\le 1,~|\Im\,s|\le T\}$,
counted with multiplicity.  The following ``log-free'' zero-density
estimate is due to Linnik~\cite{Linnik1}.

\begin{lemma}
\label{lem:fortune}
There is an effectively computable constant $c_1>0$ such that the bound
$$
\sum_{\chi\,(\text{\rm mod}~q)}N(\sigma,T,\chi)\ll (qT)^{c_1(1-\sigma)}
$$
holds uniformly for $q\ge 1$, $\sigma\in[0,1]$ and $T\ge 1$.
\end{lemma}

For our proof of Theorem~\ref{thm:b-main}, putative Siegel zeros have an impact,
and exceptional moduli must be taken into account; see Davenport~\cite[\S14]{Daven}
for a general background on exceptional moduli.  For the purposes of this paper,
we need only the following specialized result, which is
a quantitative version of the Deuring-Heilbronn phenomenon
(see Linnik~\cite{Linnik2}); for a more
general statement, we refer the reader to Davenport~\cite[\S\S13--14]{Daven}
and Knapowski~\cite{Knap} (see also Gallagher~\cite{Gallag}).

\begin{lemma}
\label{lem:exceptional}
There exist positive constants $c_2,c_3$ with the following property.
Let $\chi$ be a primitive Dirichlet character modulo $q$, where $q>1$.
Then $L(s,\chi)$ has at most one zero $\varrho=\beta+i\gamma$ such that
$$
\beta>1-\frac{c_2}{\log q}\mand |\gamma|\le q.
$$
If there is such an exception, then the exceptional zero is real,
simple and unique.  Moreover, denoting by $\beta_1$ the exceptional zero,
we have $L(s,\chi)\ne 0$ if $s=\sigma+it\ne\beta_1$ satisfies
$$
\sigma>1-\frac{c_3}{\log q}\,\log\bigg(\frac{ec_2}{(1-\beta_1)\log q}\bigg)
\mand |t|\le q.
$$
\end{lemma}

The next result, which may be of independent interest,
is a variant of Montgomery and Vaughan~\cite[Exercise~2, p.~382]{MontVau};
our proof uses ideas of Gallagher (see~\cite[\S4]{Gallag}).

\begin{lemma}
\label{lem:fame}
There is an effectively computable constant $c_4>0$ with
the following property.
Let $\chi$ be a primitive Dirichlet character modulo $q$, where $q>1$, and put
$$
\psi(x,\chi):=\sum_{n\le x}\chi(n)\Lambda(n)\qquad(x>0),
$$
where $\Lambda$ is the von Mangoldt function.
For any $c>0$ there is a constant $K=K(c)$
such that the estimate
$$
\psi(x,\chi)=-\sum_\varrho\widetilde m(\varrho)\frac{x^\varrho}{\varrho}
+O\bigg(x\,\exp\bigg(-c\,\frac{\log x}{\log q}\bigg)\bigg)
$$
holds uniformly provided that
\begin{equation}\label{eq:bounds}
\exp(K\sqrt{\log x})\le q\le x^{c_4},
\end{equation}
where the sum is taken over distinct zeros $\varrho=\beta+i\gamma$ of $L(s,\chi)$ for which $\beta>1-2c/\log q$ and $|\gamma|\le q$,
$\widetilde m(\varrho)$ is the multiplicity of any such zero,
and the implied constant depends only on $c$.
\end{lemma}

\begin{proof} 
Let $c_4:=\min\{\frac34,(4c_1)^{-1}\}$, where $c_1$ is the constant described in Lemma~\ref{lem:fortune}.

We have by Davenport~\cite[\S19]{Daven} (with $T:=q$):
\begin{equation}\label{ssylka}
\psi(x,\chi)=-\sum_{\varrho\in\sZ}\widetilde m(\varrho)
\frac{x^\varrho}{\varrho}+R(x,q),
\end{equation}
where
$$
|R(x,q)| \ll xq^{-1} \log^2{qx} +x^{1/4}\log x
$$
and $\sZ$ denotes the set of nontrivial zeros
$\varrho=\beta+i\gamma$ of $L(s,\chi)$ such that $0\le\beta\le 1$
and $|\gamma|\le q$. 
(Recall that if $\chi(-1)=1$,  the Dirichlet $L$-function $L(s, \chi)$ vanishes at $s=0$; however, this trivial zero $\varrho=0$ is not included in the sum).

Since by hypothesis $
q\le x^{c_4} $  and  $ c_4\le \frac34$, we have 
$$
 \log^2{qx}=O\left ( \log^2 x\right )\mand x^{1/4}\log x=O\bigg(\frac{x\log^2x}{q}\bigg),
$$
and therefore
$$
|R(x,q)|=O\bigg(\frac{x\log^2x}{q}\bigg),
$$
which together with \eqref{ssylka} proves that 
\begin{equation}\label{razraz}
\psi(x,\chi)=-\sum_{\varrho\in\sZ}\widetilde m(\varrho)
\frac{x^\varrho}{\varrho}+O\bigg(\frac{x\log^2x}{q}\bigg)\qquad (x\to \infty)
\end{equation}
(or $q\to \infty$, cf.\ \eqref{eq:bounds}).

We observe that for any fixed $K>\sqrt{c}$  we have the following estimate 
$$\frac{x\log^2x}{q}=O(E)\qquad
(q\ge\exp(K\sqrt{\log x})),
$$
where
$$
E:=x\,\exp(-c(\log x)/\log q).
$$
Thus \eqref{razraz} yields the representation  
$$
\psi(x,\chi)
=-\sum_{\varrho\in\sZ}\widetilde m(\varrho)\frac{x^\varrho}{\varrho}+O(E).
$$

Put $\eta:=2c/\log q$, and let
$\sZ_1$ and $\sZ_2$ be the set of zeros in $\sZ$ that satisfy
$\beta\le 1-\eta$ and $\beta>1-\eta$, respectively.  To prove the lemma, we need to show that
\begin{equation}\label{raz}
\sum_{\varrho\in\sZ_1}\widetilde m(\varrho)\frac{x^\varrho}{\varrho}\ll E.
\end{equation}
To do this, choose some $\beta_0\in(0,\frac12)$
and split $\sZ$ into disjoint subsets
$$
\sZ=\sL\cup\sR,
$$
where $\sL$ and $\sR$ denote the set of zeros in $\sZ_1$ that satisfy the inequalities $0<\beta\le\beta_0$ and
$\beta_0<\beta\le1-\eta$, respectively.  Define
the corresponding zero-counting functions
$$
N_\sL(\sigma,q,\chi):=N(\sigma,q,\chi)-N\left (\beta_0,q,\chi\right ) \qquad (\sigma\in\left [0,\beta_0\right ])
$$
and 
$$
N_\sR(\sigma,q,\chi):=N(\sigma,q,\chi)-N(1-\eta,q,\chi)
\qquad(\sigma\in[\beta_0,1-\eta]),
$$
where as above $N(\sigma,q,\chi)$ denotes the number of zeros
of $L(s,\chi)$ in the region
$$
\{s\in\C:\sigma\le\Re\,s\le 1,~|\Im\,s|\le q\},
$$
counted with multiplicity. 

To bound $\sum_{\varrho\in\sR}\widetilde m(\varrho)\frac{x^\varrho}{\varrho}$ we begin by observing that
\begin{align*}
 \min_{\varrho\in \sR}|\varrho|\cdot \bigg|\sum_{\varrho\in\sR}\widetilde m(\varrho)\frac{x^\varrho}{\varrho}\bigg|
&\ll\sum_{\varrho\in\sR}\widetilde m(\varrho)x^{\Re(\rho)}
=-\int_{\beta_0}^{1-\eta}x^\alpha\,dN_\sR(\alpha,q,\chi)\\
&=x^{\beta_0}N_\sR(\beta_0,q,\chi)
+\log x\int_{\beta_0}^{1-\eta}x^\alpha N_\sR(\alpha,q,\chi)\,d\alpha\\
&\le  x^{\beta_0}N(0,q,\chi)+\log x\int_{0}^{1-\eta}x^\alpha N(\alpha,q,\chi)\,d\alpha.
\end{align*}
 Since $c_4\le (4c_1)^{-1}$ we have by Lemma~\ref{lem:fortune}:
$$
N(\alpha,q,\chi)\ll q^{2c_1(1-\alpha)}\le x^{1/2(1-\alpha)};
$$
thus,
\begin{align*}
 x^{\beta_0}N(0,q,\chi)+\log x\int_{0}^{1-\eta}x^\alpha N(\alpha,q,\chi)\,d\alpha
 &\ll
 x^{\beta_0+1/2}+\log x\int_0^{1-\eta}x^{\frac12(1+\alpha)}\,d\alpha
\\
&\ll  x^{\beta_0+1/2}+ x^{1-\eta/2}.
\end{align*}
Consequently, we have
$$
\bigg|\sum_{\varrho\in\sR}\widetilde m(\varrho)\frac{x^\varrho}{\varrho}\bigg|\ll
\frac{1}{\min_{\varrho\in \sR}|\varrho|}
(x^{\beta_0+1/2}+x^{1-\eta/ 2}).
$$
Since
$$
 \min_{\varrho\in \sR}|\varrho|\ge \frac1\beta_0>2
$$
we deduce that
\begin{equation}\label{tri}
\bigg|\sum_{\varrho\in\sR}\widetilde m(\varrho)\frac{x^\varrho}{\varrho}\bigg|\ll x^{1-\eta/ 2}=E=O(E).
\end{equation}

To estimate the sum  $\sum_{\varrho\in\sL}\widetilde m(\varrho)\frac{x^\varrho}{\varrho}$ we proceed in a similar way,
assuming initially that the character $\chi$ is not exceptional
(that is, the function $L(s,\chi)$ has no Siegel zero).
We have 
\begin{align*}
\min_{\varrho\in \sL}|\varrho|\cdot
\bigg|\sum_{\varrho\in\sL}\widetilde m(\varrho)\frac{x^\varrho}{\varrho}\bigg|
&\ll\sum_{\varrho\in\sL}\widetilde m(\varrho)x^{\beta(\varrho)}
=-\int_0^{\beta_0}x^\alpha\,dN_\sL(\alpha,q,\chi)\\
&=N_\sL(0,q,\chi)
+\log x\int_0^{\beta_0}x^\alpha N_\sL(\alpha,q,\chi)\,d\alpha\\
&\le N(0,q,\chi)+\log x\int_0^{\beta_0}x^\alpha N(\alpha,q,\chi)\,d\alpha\\
&\ll x^{1/2}+x^{(1+\beta_0)/2}\ll x^{(1+\beta_0)/2}.
\end{align*}
Applying Lemma \ref{lem:exceptional} and using the functional equation for the $L$-function, we have the lower bound
$$
\min_{\varrho\in \sL}|\varrho|\ge \frac{c_2}{\log q},
$$
and hence
\begin{equation}\label{dva}
\bigg|\sum_{\varrho\in\sL}\widetilde m(\varrho)\frac{x^\varrho}{\varrho}\bigg|\ll\log q \cdot x^{3/4}=O(E),
\end{equation}
where we have taken into account that $q\le x^{c_4}$ and that
$
\frac{1+\beta_0}{2}<\frac34.
$
Combining \eqref{dva} and \eqref{tri} we obtain \eqref{raz}
in this case.

To treat the case in which $\chi$ is exceptional, suppose now
that $L(\beta^*,\chi)=0$ with $\beta^*$ being the exceptional
zero.  Since $\chi$ is a primitive character, one can use the functional equation for  $L$-functions to conclude that 
$L(\delta^*, \chi)=0
$ where $\delta^*=1-\beta^*$ (see, e.g.,  \cite[\S9, eq.(8)]{Daven} and  \cite[\S9, eq.(11)]{Daven} if $\chi(-1)=1$ and  $\chi(-1)=-1$, respectively).  Then
$$
\left |\sum_{\varrho\in\sL}\widetilde m(\varrho)\frac{x^\varrho}{\varrho}\right |\le \frac{x^{\delta^*}}{\delta^*}
+\left |\sum_{\varrho\in\sL\setminus \{\delta^*\}}\widetilde m(\varrho)\frac{x^\varrho}{\varrho}\right |.
$$
Using Lemma \ref{lem:exceptional} one concludes as above that 
\begin{equation}\label{dva2}
\left|\sum_{\varrho\in\sL\setminus \{\delta^*\}}\widetilde m(\varrho)\frac{x^\varrho}{\varrho}\right|=O(E).
\end{equation}
For the remaining term we use the estimate
$$
\delta^*=1-\beta^*>Cq^{-\frac12}
\begin{cases}
1 & \text{if } \chi(-1)=-1,\\
\log q &\text{if } \chi(-1)=1,
\end{cases}
$$
with some $C>0$ (see, e.g., \cite{Gold}), which yields
for any fixed $\varepsilon>0$:
$$
 \frac{x^{\delta^*}}{\delta^*}=O(\sqrt{q}\cdot x^\varepsilon)=O(x^{c_4/2}\cdot x^\varepsilon)=O(E).
$$
Combining this estimate with \eqref{dva2} and \eqref{tri}
we also obtain \eqref{raz} in the case that $\chi$ is exceptional.
\end{proof}

Finally, we need the following statement.

\begin{lemma}
\label{lem:gift}
Let $\cU$ be a finite set of complex numbers.  For any nonzero complex
numbers $c_u$ one can find arbitrarily large values of $\theta$
for which the function $f(\theta):=\sum_{u\in\cU}c_u e^{-u\theta}$
satisfies the lower bound $|f(\theta)|\ge Ce^{-\mu \theta}$, where
$\mu:=\min_{u\in\cU}\{\Re\,u\}$ and $C$ is a positive constant
depending only on $f$.
\end{lemma}

\begin{proof}
Replacing $f(\theta)$ with
$e^{u_0\theta}f(\theta)$, where $u_0$ denotes any fixed
element of $\cU$ for which $\Re\,u_0=\mu$, we can assume
without loss of generality that $\mu=0$.
Moreover, denoting by $\cU_+$ the set of $u\in\cU$ with $\Re\,u>0$,
we clearly have
$\sum_{u\in\cU_+}c_u e^{-u\theta}=o(1)$ as $\theta\to\infty$;
hence, we can also assume that $\Re\,u=0$ for all $u\in\cU$.
With these assumptions,
the lemma is a consequence of Wiener's Lemma:
\begin{equation}
\label{eq:wiener}
\lim_{T\to\infty}\frac{1}{T}\int_0^T |f(\theta)|^2\,d\theta
=\sum_{u\in\cU}|c_u|^2.
\end{equation}
Indeed, the premise that $\limsup_{\theta\to\infty}|f(\theta)|=0$ leads to
$$
\lim_{T\to\infty}\frac{1}{T}\int_0^T |f(\theta)|^2\,d\theta
=\lim_{\theta\to\infty}|f(\theta)|^2=0,
$$
which is impossible in view of \eqref{eq:wiener};
therefore, $\limsup_{\theta\to\infty}|f(\theta)|>0$,
which completes the proof.
\end{proof}

\subsection{A relation involving $\delta$}

Thanks to Hildebrand~\cite{Hild} it is known that for every $\eps>0$
there is a number $p_0(\eps)\ge 2$ such that
$$
\bigg|\sum_{n\le X}(n|p)\bigg|\le\eps X
\qquad (p\ge p_0(\eps),~X\ge p^{1/4}).
$$
The next statement is an immediate consequence of
Hildebrand's result.

\begin{lemma}
\label{lem:hildebrand} The estimate
$$
\big|\sN_p(X)\big|=(1/2+o(1))X\qquad(p\to\infty)
$$
holds for all $X\ge p^{1/4}$, where $\sN_p$ is given by \eqref{eq:NpXdefi}, and the function implied by $o(1)$
depends only on $p$.
\end{lemma}

In what follows, let $C$ be a large positive number.  All constants
implied by the symbols $O$ and $\ll$ may depend on $\kappa,\lambda,\distr,C$ but
are absolute otherwise.
The symbol $\error$ in any expression below indicates an error term
that tends to zero as $p$ tends to infinity \emph{within the set $\sP$}.
Any function of $p$ implied by $\error$ may depend on
$\kappa,\lambda,\distr,C$ but is independent of all other parameters.

For every prime $p\in\sP$ let $\sK_p$ denote the set of squarefree
integers $k>1$ with the property that $(q|p)=-1$ for all primes $q$
dividing $k$. The next result is based on the inclusion-exclusion principle.

\begin{lemma}
\label{lem:inc-excl} Uniformly for $\theta\in[0,C]$ we have
$$
\big|\sN_p(p^\theta)\big|=\sum_{\substack{k\le
p^\theta\\k\in\sK_p}}(-1)^{\omega(k)+1}\big|\{m\le
p^\theta/k:(m|p)=-(k|p)\}\big|+\error p^\theta,
$$
where $\omega(k)$ is the number of distinct prime divisors of $k$.
\end{lemma}

\begin{proof}
For each $p\in\sP$ let $\sA_p$ denote the set of ordered pairs given by
$$
\sA_p:=\{(m,k):k\in\sK_p,~m\le p^\theta/k,~(m|p)=-(k|p)\}.
$$
Then
\begin{equation}
\label{eq:Ap1} \sum_{(m,k)\in\sA_p}(-1)^{\omega(k)+1}
=\sum_{\substack{k\le
p^\theta\\k\in\sK_p}}(-1)^{\omega(k)+1}
\big|\{m\le p^\theta/k:(m|p)=-(k|p)\}\big|.
\end{equation}
Next, split $\sN_p(p^\theta)$ into a disjoint union
$\sN_1\cup\sN_2$, where
$$
\sN_1:=\{n\in\sN_p(p^\theta):q^2\nmid n\text{~if~}(q|p)=-1\},
$$
and $\sN_2:=\sN_p(p^\theta)\setminus\sN_1$.
Since $n_0(p)>p^\kappa$ and $\kappa>0$ we have
$$
|\sN_2|\le\sum_{p^\kappa<q\le p^\theta}
\big|\{n\le p^\theta:q^2\mid n\text{~and~}(q|p)=-1\}\big|\le
p^\theta\sum_{q>p^\kappa}q^{-2}=\error p^\theta,
$$
and therefore
\begin{equation}
\label{eq:Ap2} \big|\sN_p(p^\theta)\big|=|\sN_1|+\error p^\theta.
\end{equation}
Each number $n\in\sN_1$ can be factored as $n^+n^-$, where
$$
n^+:=\prod_{\substack{q^\alpha\|n\\ (q|p)=+1}}q^\alpha\mand
n^-:=\prod_{\substack{q\,\mid\,n\\ (q|p)=-1}}q.
$$
Let $r_j(n)$ denote the number of pairs $(m,k)\in\sA_p$ such that
$mk=n$ and $\omega(k)=j$.  Then
$$
r_j(n)=\big|\{k>1:k\,\mid\,
n^-,~\omega(k)=j\}\big|=\binom{\omega(n^-)}{j}.
$$
Hence, denoting by $\sB_p$ the subset of $\sA_p$ consisting
of pairs $(m,k)$ for which $mk\in\sN_1$, we have
\begin{equation}
\label{eq:Ap3} \sum_{(m,k)\in\sB_p}(-1)^{\omega(k)+1}=
\sum_{n\in\sN_1}\sum_{j=1}^{\omega(n^-)}(-1)^{j+1}r_j(n)=|\sN_1|
\end{equation}
since each inner sum is
$$
\sum_{j=1}^{\omega(n^-)}(-1)^{j+1}\binom{\omega(n^-)}{j}=1.
$$
Finally, we observe that the set $\sC_p:=\sA_p\setminus\sB_p$ consists
of pairs $(m,k)\in\sA_p$ such that $q^2\mid mk$ for some prime
nonresidue $q$.  Fixing $\eps:=\kappa/(2C)$
and using the divisor bound $\sum_{k\,\mid\, n}1\ll n^\eps$ for
all $n\in\N$, for any $\theta\in[0,C]$ we derive that
\begin{equation*}
\begin{split}
|\sC_p|&\le\sum_{p^\kappa<q\le
p^\theta}\sum_{\substack{n\le p^\theta\\
q^2\,\mid\,n}}\sum_{k\,\mid\,n}1
\ll p^{\theta\eps}\sum_{p^\kappa<q\le
p^\theta}\sum_{\substack{n\le p^\theta\\
q^2\,\mid\,n}}1\\
&\le p^{\theta(1+\eps)}\sum_{q>p^\kappa}q^{-2} \ll
p^{\theta(1+\eps)-\kappa}\le p^{\theta-\kappa/2}
=\error p^\theta.
\end{split}
\end{equation*}
Using this result together with~\eqref{eq:Ap2} and~\eqref{eq:Ap3}
we deduce that
\begin{equation*}
\begin{split}
\sum_{(m,k)\in\sA_p}(-1)^{\omega(k)+1}
=\sum_{(m,k)\in\sB_p}(-1)^{\omega(k)+1}+\error p^\theta
=\big|\sN_p(p^\theta)\big|+\error p^\theta.
\end{split}
\end{equation*}
In view of~\eqref{eq:Ap1}, this completes the proof.
\end{proof}

Next, using \eqref{eq:Np-est} we see that for fixed $k\in\sK_p$
and uniformly for $\vartheta\in[0,C]$ we have
\begin{equation}
\label{eq:chess}
\big|\{n\le p^\vartheta:(n|p)=-(k|p)\}\big|
=(\delta_{k,p}(\vartheta)+\error)p^\vartheta,
\end{equation}
where
$$
\delta_{k,p}:=\left\{
\begin{array}{ll}
    \delta & \quad\hbox{if $(k|p)=+1$,} \\
    1-\delta & \quad\hbox{if $(k|p)=-1$\, .} \\
\end{array}
\right.
$$
Hence, from Lemma~\ref{lem:inc-excl} we deduce the estimate
$$
\big|\sN_p(p^\theta)\big|
=\sum_{\substack{k\le p^\theta\\k\in\sK_p}}
(-1)^{\omega(k)+1}\big(\delta_{k,p}
\big(\theta-\tfrac{\log k}{\log p}\big)+\error\big)p^\theta k^{-1}
+\error p^\theta.
$$
Mertens' theorem yields the bound
\begin{equation}
\label{eq:Mertens}
\sum_{\substack{k\le p^\theta\\k\in\sK_p}}k^{-1}
\le\bigg(\,\sum_{p^{\kappa}<q\le
p^\theta}q^{-1}\bigg)^{\fl{\theta/\kappa}}
\le(\log(C/\kappa)+O(1))^{C/\kappa}=O(1),
\end{equation}
where $\fl{\cdot}$ is the floor function,
and therefore
$$
\big|\sN_p(p^\theta)\big|=p^\theta \sum_{\substack{k\le
p^\theta\\k\in\sK_p}}(-1)^{\omega(k)+1}
\delta_{k,p}\big(\theta-\tfrac{\log k}{\log p}\big)k^{-1}+\error
p^\theta.
$$
Dividing both sides by $p^\theta$, using \eqref{eq:chess}, and taking
into account the fact that $\delta(\theta-u)=0$ for $u\ge\theta-\kappa$,
we derive the relation
\begin{equation}
\label{eq:delta-relation} \delta(\theta)=\sum_{\substack{k\le
p^\theta\\k\in\sK_p\\\omega(k)\text{~odd}}}k^{-1}-
\sum_{\substack{k\le p^{\theta-\kappa}\\k\in\sK_p}}\delta
\big(\theta-\tfrac{\log k}{\log p}\big)k^{-1}+\error,
\end{equation}
which holds uniformly for all $\theta\in[0,C]$.

\subsection{The functions $\{S_{p,j}\}$}

Next, we study the functions defined by
\begin{equation}
\label{eq:Spj-defn}
S_{p,j}(\theta):=\sum_{\substack{k\le p^\theta,~k\in\sK_p\\
\omega(k)=j}}k^{-1}\qquad(p\in\sP,~j\ge 0,~\theta\ge 0).
\end{equation}
Each function $S_{p,j}$ is a nondecreasing step function of
bounded variation on any finite interval.  From \eqref{eq:Mertens}
we see that the bound
\begin{equation}
\label{eq:Spj-bound-above}
S_{p,j}(\theta)=O(1)
\end{equation}
holds uniformly for $p\in\sP$, $j\ge 0$ and $\theta\in[0,C]$.  Note that
for $j\in\N$ we have
\begin{equation}
\label{eq:Spj-vanishes}
S_{p,j}(\theta)=0\qquad(\theta\in [0,j\kappa])
\end{equation}
since every integer $k$ occurring in the sum \eqref{eq:Spj-defn}
has $j$ distinct prime factors, each of size at least $n_0(p)>p^\kappa$.

\begin{lemma}
\label{lem:Sp1} Uniformly for $\theta\in[\kappa,C]$ we have
\begin{equation}
\label{eq:Sp1-expansion}
S_{p,1}(\theta)=\delta(\theta)-\sum_{\text{\rm odd~}j\ge 3}
S_{p,j}(\theta)+\frac12\sum_{j\in\N}
\int_0^{\theta-\kappa} S_{p,j}(u)\,\distr(\theta-u)\,du+\error.
\end{equation}
\end{lemma}

\begin{proof}
Since $\delta(\theta-u)=0$ for $u\in[\theta-\kappa,\theta]$,
using \eqref{eq:intrho} we have for all $j\in\N$:
\begin{align*}
&\sum_{\substack{k\le
p^{\theta-\kappa},~k\in\sK_p\\\omega(k)=j}}\delta
\big(\theta-\tfrac{\log
k}{\log p}\big)k^{-1}
=\int_0^{\theta-\kappa}\delta(\theta-u)\,dS_{p,j}(u)\\
&\qquad\qquad=-\int_0^{\theta-\kappa} S_{p,j}(u)\,d\delta(\theta-u)
=\frac12\int_0^{\theta-\kappa} S_{p,j}(u)\,\distr(\theta-u)\,du,
\end{align*}
where the integrals are of Riemann-Stieltjes type (note that these
integrals are well-defined since
$S_{p,j}$ is of bounded variation). The result now follows by
inserting this expression into~\eqref{eq:delta-relation}.
\end{proof}

\medskip

\begin{remark} Using \eqref{eq:Spj-vanishes} one sees that both sums
in~\eqref{eq:Sp1-expansion} have only finitely many nonzero terms,
the number of such terms being bounded by a constant that depends
only on $\kappa,\lambda,\distr,C$.
\end{remark}

\medskip

\begin{lemma}
\label{lem:Spj}
Uniformly for $j\in\N$ and $\theta\in[0,C]$ we have
$$
S_{p,j}(\theta)=\frac{1}{j}\sum_{\substack{q\le p^\theta\\
(q|p)=-1}}S_{p,j-1}\big(\theta-\tfrac{\log q}{\log
p}\big)\,q^{-1}+\error.
$$
\end{lemma}

\begin{proof}
For any natural number $k$, let $\omega(k)$ be the number of distinct
prime divisors of $k$, and let $\Omega(k)$ be the total number of primes
dividing $k$, counted with multiplicity.

We first show that for $j\in\N$ and $\theta\in[0,C]$ the estimate
\begin{equation}
\label{eq:Spj-relation} S_{p,j}(\theta)=
\frac{1}{j!}\sum_{\substack{(q_1,\ldots,q_j)\\
(q_i|p)=-1~\forall\,i\\q_1\cdots q_j\le
p^\theta}}\frac{1}{q_1\cdots q_j}+\error
\end{equation}
holds uniformly.  We can assume that $j\le C/\kappa$, for otherwise
the sum is empty and thus \eqref{eq:Spj-relation}
follows immediately from \eqref{eq:Spj-vanishes}.

For each $p\in\sP$ let $\cL_p$ be the set of integers $k>1$
such that $(q|p)=-1$ for every prime divisor $q$ of $k$; note
that a number $k\in\cL_p$ lies in $\sK_p$ if and only if $k$ is
squarefree. Let $r_j(k)$ be the number of
ordered $j$-tuples $(q_1,\ldots,q_j)$ of primes such that
$q_1\cdots q_j=k$ and $(q_i|p)=-1$ for each~$i$. Then
\begin{itemize}
\item $0\le r_j(k)\le j!$;

\item $r_j(k)\ne 0$ if and only if $k\in\cL_p$ and $\Omega(k)=j$;

\item $r_j(k)=j!$ if and only if $k\in\sK_p$ and $\omega(k)=j$.
\end{itemize}
These properties imply that
$$
\sum_{\substack{(q_1,\ldots,q_j)\\
(q_i|p)=-1~\forall\,i\\q_1\cdots q_j\le
p^\theta}}\frac{1}{q_1\cdots q_j}=j!\sum_{\substack{k\le
p^\theta,~k\in\sK_p\\\omega(k)=j}}k^{-1}+\sum_{\substack{k\le
p^\theta,~k\in\cL_p\setminus\sK_p\\\Omega(k)=j}}r_j(k)k^{-1}.
$$
Dividing both sides by $j!$ and recalling that $j\le C/\kappa$, we see that
\begin{equation*}
\begin{split}
&\bigg|S_{p,j}(\theta)
-\frac{1}{j!}\sum_{\substack{(q_1,\ldots,q_j)\\
(q_i|p)=-1~\forall\,i\\q_1\cdots q_j\le
p^\theta}}\frac{1}{q_1\cdots q_j}\bigg|\le
\fl{C/\kappa}!\sum_{\substack{k\le
p^\theta\\k\in\cL_p\setminus\sK_p}}k^{-1}\ll
\sum_{p^\kappa<q\le
p^\theta}\sum_{\substack{n\le p^\theta\\q^2\,\mid\,n}}n^{-1}\\
&\qquad\qquad\ll \sum_{q>p^\kappa}q^{-2}\,\log(p^\theta)\ll
p^{-\kappa}\log(p^\theta)=\error,
\end{split}
\end{equation*}
which yields \eqref{eq:Spj-relation}.

To complete the proof, we suppose that $\theta\in[0,C]$ and apply~\eqref{eq:Spj-relation}
with both $j$ and $j-1$  to derive that
\begin{equation*}
\begin{split}
S_{p,j}(\theta)&=
\frac{1}{j!}\sum_{\substack{(q_1,\ldots,q_j)\\
(q_i|p)=-1~\forall\,i\\q_1\cdots q_j\le
p^\theta}}\frac{1}{q_1\cdots
q_j}+\error\\
&=\frac{1}{j!}\sum_{\substack{q_1\le p^\theta\\
(q_1|p)=-1}}\frac{1}{q_1}\sum_{\substack{(q_2,\ldots,q_j)\\
(q_i|p)=-1~\forall\,i\\q_2\cdots q_j\le
p^\theta/q_1}}\frac{1}{q_2\cdots
q_j}+\error\\
&=\frac{1}{j}\sum_{\substack{q\le p^\theta\\
(q|p)=-1}}\big(S_{p,j-1}\big(\theta-\tfrac{\log q}{\log
p}\big)+\error\big)q^{-1}+\error\\
&=\frac{1}{j}\sum_{\substack{q\le p^\theta\\
(q|p)=-1}}S_{p,j-1}\big(\theta-\tfrac{\log q}{\log
p}\big)q^{-1}+\error,
\end{split}
\end{equation*}
where we have used Mertens' theorem in the final step.
\end{proof}

\subsection{The functions $\{S_j\}$}

\begin{proposition}
\label{prop:continuity} We have
\begin{itemize}
\item[$(i)$] The limits
$$
S_j(\theta):=\lim_{\substack{p\to\infty\\p\in\sP}}S_{p,j}(\theta)
\qquad(j\ge 0,~\theta\in[0,C])
$$
exist, and the estimate $S_j(\theta)=S_{p,j}(\theta)+\error$ holds uniformly
for $j\ge 0$ and $\theta\in[0,C]$;

\item[$(ii)$] Each function $S_j$ is continuous at $\theta\in[0,C]$.
\end{itemize}
\end{proposition}

\begin{proof}
For $j=0$ there is nothing to prove, so we assume that $j\in\N$ in what follows.
For each $\ell\in\N$ let $\cI_\ell$ be the interval $[0,C_\ell]$,
where $C_\ell:=\ell\kappa$. By induction on $\ell$
we show that both statements hold when $C=C_\ell$, the case $\ell=1$ being
an immediate consequence of \eqref{eq:Spj-vanishes}.

Now suppose that $(i)$ and $(ii)$ hold with $\ell\in\N$ and $C=C_\ell$,
and let $j\in\N$ and $\theta\in\cI_{\ell+1}$ be fixed.
If $\theta<j\kappa$, then $S_{p,j}(\theta)=S_j(\theta)=0$ for all $p\in\sP$
by~\eqref{eq:Spj-vanishes}; thus, we can assume without loss of generality
that $\theta\ge j\kappa$.

First, consider the case that $j\ge 2$.
As $n_0(p)>p^\kappa$, Lemma~\ref{lem:Spj} implies that
$$
S_{p,j}(\theta)=\frac{1}{j}\sum_{\substack{p^\kappa<q\le p^\theta\\
(q|p)=-1}}q^{-1}S_{p,j-1}\big(\theta-\tfrac{\log q}{\log
p}\big)+\error.
$$
For any prime $q\in(p^\kappa,p^\theta]$ we have $\theta-\tfrac{\log q}{\log
p}\in\cI_\ell\,$; therefore, using $(i)$ with $C=C_\ell$ together
with~\eqref{eq:Spj-bound-above} and~\eqref{eq:Spj-vanishes}
we derive that
\begin{equation*}
\begin{split}
S_{p,j}(\theta)&=\frac{1}{j}\sum_{\substack{p^\kappa<q\le p^\theta\\
(q|p)=-1}}q^{-1}\(S_{j-1}\big(\theta-\tfrac{\log q}{\log
p}\big)+\error\)+\error\\
&=\frac{1}{j}\sum_{\substack{p^\kappa<q\le p^\theta\\
(q|p)=-1}}q^{-1}S_{j-1}\big(\theta-\tfrac{\log q}{\log
p}\big)+\error\\
&=\frac{1}{j}\sum_{\substack{p^\kappa<q\le p^{\theta-(j-1)\kappa}\\
(q|p)=-1}}q^{-1}S_{j-1}\big(\theta-\tfrac{\log q}{\log
p}\big)+\error,
\end{split}
\end{equation*}
where we have used \eqref{eq:Spj-vanishes} 
with $j-1$ in place of $j$ to reduce the range of $q$ in the last sum.
By $(ii)$ with $C=C_\ell$, the function
$S_{j-1}$ is continuous on $\cI_\ell$, and therefore the last sum
can be expressed as a Riemann-Stieltjes integral:
\begin{equation*}
\begin{split}
\int_\kappa^{\theta-(j-1)\kappa}
S_{j-1}(\theta-u)\,dS_{p,1}(u)=-\int_\kappa^{\theta-(j-1)\kappa}
S_{p,1}(u)\,dS_{j-1}(\theta-u).
\end{split}
\end{equation*}
Since $j\ge 2$ and $\theta\le C_{\ell+1}$,
one verifies that $u\in \cI_\ell$ for all values
of $u$ in these integrals; hence, using $(ii)$ with $C=C_\ell$ along
with~\eqref{eq:Spj-bound-above} and~\eqref{eq:Spj-vanishes} we have
\begin{equation*}
\begin{split}
\int_\kappa^{\theta-(j-1)\kappa} S_{p,1}(u)\,dS_{j-1}(\theta-u)
&=\int_\kappa^{\theta-(j-1)\kappa} (S_1(u)+\error)\,dS_{j-1}(\theta-u)\\
&=\int_\kappa^{\theta-(j-1)\kappa} S_1(u)\,dS_{j-1}(\theta-u) +\error\\
&=-\int_\kappa^{\theta-(j-1)\kappa} S_{j-1}(\theta-u)\,dS_1(u) +\error.
\end{split}
\end{equation*}
Putting everything together, we have shown that
$$
S_{p,j}(\theta)=\frac{1}{j}\int_\kappa^{\theta-(j-1)\kappa}
S_{j-1}(\theta-t)\,dS_1(t) +\error.
$$
This proves $(i)$ for $C=C_{\ell+1}$ in the case that $j\ge 2$.
Considering separately the cases $\theta\le j\kappa$ and
$\theta>j\kappa$, we have established the following relation:
$$
S_j(\theta)=\frac{1}{j}\int_0^\theta
S_{j-1}(\theta-t)\,dS_1(t) \qquad(j\ge 2,~\theta\in\cI_{\ell+1}).
$$

Next, we prove $(ii)$ for $j\ge 2$ and $\theta\in\cI_{\ell+1}$.
Let $\eps>0$ be given, and suppose that $p\in\sP$ is large enough
so that
$$
p^{-\kappa}<\eps/3\mand
\big|S_j(\theta)-S_{p,j}(\theta)\big|<\eps/3.
$$
Each $S_{p,j}$ is a step function, and the size of the
step at an integer $k>n_0(p)$ is $k^{-1}<p^{-\kappa}<\eps/3$.
Since the steps occur on a discrete subset of the real line,
it follows that the bound
$$
\big|S_{p,j}(\theta)-S_{p,j}(\theta')\big|<\eps/3
$$
holds for all $\theta'$ in a small neighborhood of $\theta$.
Hence, if $\theta'\in\cI_{\ell+1}$ is sufficiently close to $\theta$, then
$$
\big|S_j(\theta)-S_j(\theta')\big|
\le \big|S_j(\theta)-S_{p,j}(\theta)\big|+
\big|S_{p,j}(\theta)-S_{p,j}(\theta')\big|+
\big|S_{p,j}(\theta')-S_j(\theta')\big|<\eps.
$$
Therefore, $S_j$ is continuous at $\theta$.

It remains to verify $(i)$ and $(ii)$ for the case $j=1$. Since
$\theta\ge\kappa$, we can apply Lemma~\ref{lem:Sp1}; in view of the
remark that follows Lemma~\ref{lem:Sp1}, both sums in \eqref{eq:Sp1-expansion}
have at most finitely many nonzero terms, hence we derive that
$$
S_{p,1}(\theta)=\delta(\theta)-\sum_{\text{\rm odd~}j\ge 3}
S_j(\theta)+\frac12\sum_{j\in\N}
\int_0^{\theta-\kappa} S_j(u)\,\distr(\theta-u)\,du+\error,
$$
which implies $(i)$ for $C=C_{\ell+1}$.  This relation can also be used to
prove $(ii)$, or one can use an argument identical to the one
given above for the case $j\ge 2$.  This completes the induction,
and the proposition is proved.
\end{proof}

The next corollary follows immediately from the statement and proof
of Proposition~\ref{prop:continuity}; we omit the details.

\begin{corollary}
For any $j\in\N$ the limit
$$
S_j(\theta):=\lim_{\substack{p\to\infty\\p\in\sP}}S_{p,j}(\theta)
$$
exists and is finite for all $\theta\ge 0$, the function $S_j$ is
continuous on $[0,\infty)$, and the following relations hold for all
$\theta\ge 0$:
\begin{align}
\label{eq:Sj-reln-one} S_j(\theta)&=\frac{1}{j}\int_0^\theta
S_{j-1}(\theta-u)\,dS_1(u),\\
\label{eq:S1-reln-two} \delta(\theta)&=
\sum_{\text{\rm odd~}j\in\N}S_j(\theta)-\frac12\sum_{j\in\N}
\int_0^{\theta} S_j(u)\,\distr(\theta-u)\,du.
\end{align}
\end{corollary}

\subsection{Laplace transform}

In what follows, we use $\sigma$ and $t$ to denote the real and
imaginary parts of the complex number $s$, respectively.

For any locally integrable function $f$   on $[0,\infty)$, we denote
by $\sL_s(f)$ the Laplace transform of $f$, i.e.,
$$
\sL_s(f):=\int_0^\infty e^{-s\theta}f(\theta)\,d\theta.
$$
Here, $s$ is a complex variable for which the integral
converges absolutely.  For an excellent account of the theory of
the Laplace transform, the reader is referred Widder~\cite{Widder}.

\begin{lemma}
The Laplace integrals $\sL_s(\distr)$, $\sL_s(\delta)$
and $\{\sL_s(S_j):j\in\N\}$ converge absolutely in the region
$\{s\in\C:\sigma>0\}$.
\end{lemma}

\begin{proof}
Since $\distr$ and $\delta$ are bounded on $[0,\infty)$, 
the integrals $\sL_s(\distr)$ and $\sL_s(\delta)$
converge when $\sigma>0$.

With Mertens' theorem we can bound
$$
S_{p,j}(\theta)\le \bigg(\,\sum_{p^\kappa<q\le p^\theta}q^{-1}\bigg)^j
\le\big(\log(\theta/\kappa)+O((\log p)^{-1})\big)^j\qquad(q\in\sP).
$$
Letting $p\to\infty$ we derive the bound
\begin{equation}
\label{eq:Sj-triv-bound} S_j(\theta)\le \(\log(\theta/\kappa)\)^j
\qquad(\theta\ge\kappa),
\end{equation}
and it follows that the integrals $\{\sL_s(S_j):j\in\N\}$ converge
when $\sigma>0$.
\end{proof}

\begin{lemma}
\label{lem:delta-laplace}
There is a constant $C>0$ that depends only on $\kappa,\lambda,\distr$
such that the inequality
\begin{equation}
\label{eq:frosty}
\max\big\{|s\,\sL_s(S_1)|,|\sL_s(\distr)|\big\}<1
\end{equation}
holds everywhere in the region
\begin{equation}
\label{eq:regionR}
\cR:=\{s\in\C:\sigma>C,~|t|<\tfrac12\kappa^{1/2}\sigma^{3/2}\},
\end{equation}
and we have
\begin{equation}
\label{eq:funny-identity}
\sL_s(S_1)=\frac{1}{2s}\sum_{n\in\N}\frac{\sL_s(\distr)^n}{n}
\qquad(s\in\cR).
\end{equation}
\end{lemma}

\begin{proof}
Fix $s=\sigma+it$ in the region \eqref{eq:regionR}, and note that
\begin{equation}
\label{eq:rudolph}
|s|^2=\sigma^2+t^2<\sigma^2(1+\tfrac14\kappa\sigma)\le\sigma^2e^{\kappa\sigma/4}.
\end{equation}
Using \eqref{eq:Sj-triv-bound} with $j=1$ we have
$$
e^{-\sigma\theta}|S_1(\theta)|\le
e^{-\sigma\theta}\log(\theta/\kappa)
\le C_1 e^{-\sigma\theta/2}\qquad(\theta\ge\kappa)
$$
for some constant $C_1$ that depends only on $\kappa,\lambda,\distr$.
Taking into account that $S_1$ vanishes on $[0,\kappa]$ we have
$$
|\sL_s(S_1)|\le C_1\int_\kappa^\infty e^{-\sigma\theta/2}\,d\theta
=2C_1\sigma^{-1}e^{-\kappa\sigma/2},
$$
which together with \eqref{eq:rudolph} yields the bound
$$
|s\,\sL_s(S_1)|^2\le 4C_1^2 e^{-3\kappa\sigma/4}<4C_1^2e^{-3\kappa C/4}.
$$
Hence, if $C>2\kappa^{-1}\log(4C_1^2)$, then $|s\,\sL_s(S_1)|<1$.

Replacing $C$ by a larger constant, if necessary, the same method shows
that $|s\,\sL_s(\delta)|<1$, using the bound $\delta(\theta)\le\tfrac12$
for all $\theta\ge\kappa$ instead of \eqref{eq:Sj-triv-bound}. In view of the fact that
$s\,\sL_s(\delta)=\tfrac12\sL_s(\distr)$, the first statement is proved.

From \eqref{eq:Sj-reln-one} it follows that
$$
\sL_s(S_j)=\frac{s\sL(S_1)\sL_s(S_{j-1})}{j}\qquad(j\in\N).
$$
By induction on $j$ this leads to the relations
$$
\sL_s(S_j)=\frac{s^{j-1}\sL_s(S_1)^j}{j!}\qquad(j\in\N).
$$
From~\eqref{eq:S1-reln-two} we further deduce that
\begin{align*}
\sL_s(\delta)
&=\sum_{\text{\rm odd~}j\in\N}\sL_s(S_j)-\frac12\sum_{j\in\N}
\sL_s(S_j*\distr)\\
&=\sum_{\text{\rm odd~}j\in\N}\frac{s^{j-1}\sL_s(S_1)^j}{j!}
-\frac{\sL_s(\distr)}{2}\sum_{j\in\N}
\frac{s^{j-1}\sL_s(S_1)^j}{j!}.
\end{align*}
Note that the sums converge absolutely by \eqref{eq:frosty}.
From the previous relation it follows that
$$
\frac{\sL_s(\distr)}{2}=s\,\sL_s(\delta)
=\frac{1}{2}\(e^{s\,\sL_s(S_1)}-e^{-s\,\sL_s(S_1)}\)
-\frac{\sL_s(\distr)}{2}\(e^{s\,\sL_s(S_1)}-1\),
$$
which leads to
$$
\sL_s(S_1)=-\frac{1}{2s}\log\(1-\sL_s(\distr)\).
$$
Using \eqref{eq:frosty} and the Maclaurin series for $\log(1-u)$
we obtain \eqref{eq:funny-identity}.
\end{proof}

\begin{proposition}
\label{prop:S1-converted} For $\theta\ge 0$ we have
\begin{equation}
\label{eq:S1-converted} S_1(\theta)=\sum_{n\in\N}
\frac{(\delta*\distr^{*(n-1)})(\theta)}{n}\qquad(\theta\ge 0).
\end{equation}
\end{proposition}

\begin{proof}
Let $T_1$ be the function of $\theta$ defined by the
right side of \eqref{eq:S1-converted}.
As $\delta$ is continuous on $[0,\infty)$, the same is true of
$\delta*\distr^{*(n-1)}$ for each~$n$. Since $\distr^{*(n-1)}$ vanishes
for  $\theta\le (n-1)\kappa$, the same is also true for
$\delta*\distr^{*(n-1)}$; this implies that $T_1$ is the sum of finitely
many continuous functions on any compact interval in $[0,\infty)$, and thus
$T_1$ is continuous on all of $[0,\infty)$.  Since
$$
\sL_s(\delta*\distr^{*(n-1)})=\sL_s(\delta)\sL_s(\distr)^{n-1}
=\frac{\sL_s(\distr)^n}{2s}
$$
for all $n\in\N$ and all $s\in\C$ with $\sigma>0$, we have
$$
\sL_s(T_1)=\sum_{n\in\N}\frac{\sL_s(\delta*\distr^{*(n-1)})}{n}
=\frac{1}{2s}\sum_{n\in\N}\frac{\sL_s(\distr)^n}{n}
=\sL_s(S_1)\qquad(s\in\cR),
$$
where $\cR$ is the region \eqref{eq:regionR}.
Now $S_1$ and $T_1$ have the same Laplace transform on $\cR$, 
hence $S_1(\theta)=T_1(\theta)$ for all $\theta\ge 0$
except possibly on a set of Lebesgue measure zero (see, for
example, Widder~\cite[Theorem~6.3]{Widder}); as both functions are
continuous, we find that $S_1=T_1$ on $[0,\infty)$, and the
proposition has been proved.
\end{proof}

\begin{corollary}
\label{cor:S_1-differentiable} The function $S_1$ is continuously
differentiable on $(\lambda,\infty)$, and
$$
S'_1(\theta)=\frac{1}{2}\sum_{n\in\N}\frac{\distr^{*n}(\theta)}{n}\qquad(\theta>\lambda).
$$
\end{corollary}

\begin{proof}
Since $\distr^{*(n-1)}$ vanishes for  $\theta\le(n-1)\kappa$,
for any constant $C>0$ the relation \eqref{eq:S1-converted} implies that
\begin{equation}
\label{eq:S1-convertedC} S_1(\theta)=\sum_{n<1+C/\kappa}
\frac{(\delta*\distr^{*(n-1)})(\theta)}{n}\qquad(0\le\theta<C).
\end{equation}
As $\delta$ is constant (hence differentiable) on $(\lambda,\infty)$, 
it follows that the function $S_1$ is differentiable on $(\lambda,C)$;
taking $C\to\infty$ we obtain the first statement of the corollary.
The second statement follows from \eqref{eq:S1-convertedC}
using the relation $\delta'=\tfrac12\distr$ and well known properties
of the Laplace integral; we omit the details.
\end{proof}

\subsection{Two expressions for $S_1'(\theta)$}

Combining Theorem~\ref{thm:c-main} and Corollary~\ref{cor:S_1-differentiable}
we obtain the following statement.

\begin{proposition}
\label{prop:firstS1'}
Let $f(k):=\widehat\distr(k)-1$, and let $c>0$ be a real number such that
$f$ does not vanish on the line $\{k\in\C:\Im\,k=-c\}$. Then
$$
S_1'(\theta)=\frac{1}{2\theta}\bigg(1+\sum_{k\in\sK} m(k)
e^{-ik\theta}+E(c, \theta)e^{-c \theta}\bigg)\qquad(\theta>\lambda),
$$
where $\sK$ is the set consisting of the $($finitely many$)$ zeros $k$ of $f$ which
lie in the strip $\Pi_c:=\{k\in\C:-c<\Im\,k<0\}$,
$m(k)$ is the multiplicity of any such zero, and
$$
E(c, \theta):=
\frac{1}{2\pi i}\int_\R \left (
\frac{\widehat\distr'(u-ic)}{1-\widehat\distr(u-ic)}\,e^{-iu\theta}
-\widehat\distr'(u-ic)\right )\,du.
$$
\end{proposition}

To obtain a second expression for $S'_1(\theta)$, we start with
the definition \eqref{eq:S1defn} and observe that
\begin{equation}
\label{eq:vida}
S_1'(\theta)=\lim_{\eps\to 0^+}\eps^{-1}(S_1(\theta+\eps)-S_1(\theta))
=\lim_{\eps\to 0^+}\eps^{-1}\lim_{\substack{p\to\infty\\ p\in\sP}}
\sum_{\substack{p^\theta<q\le p^{\theta+\eps}\\ (q|p)=-1}}q^{-1}
\end{equation}
holds for any $\theta>\lambda$.  We note that
$$
\sum_{\substack{p^\theta<q\le p^{\theta+\eps}\\ (q|p)=-1}}q^{-1}
=\frac12\sum_{p^\theta<q\le p^{\theta+\eps}}q^{-1}
-\frac12\sum_{p^\theta<q\le p^{\theta+\eps}}(q|p)q^{-1}-
\begin{cases}
\frac12\,  p^{-1} &\text{if } 1\in (\theta, \theta+\varepsilon], \\
0 &\text{otherwise},
\end{cases}
$$
and using standard techniques derive the estimates
$$
\sum_{p^\theta<q\le p^{\theta+\eps}}q^{-1}
=\eps\theta^{-1}+O((\theta\log p)^{-1}+\eps^2\theta^{-2})
$$
and
$$
\sum_{p^\theta<q\le p^{\theta+\eps}}(q|p)q^{-1}
=\sum_{p^\theta<n\le p^{\theta+\eps}}\frac{\Lambda(n)(n|p)}{n\log n}
+O(p^{-\theta});
$$
hence from \eqref{eq:vida} it follows that
\begin{equation}
\label{eq:vida2}
S_1'(\theta)=
\frac{1}{2\theta}-\frac12
\lim_{\eps\to 0^+}\eps^{-1}\lim_{\substack{p\to\infty\\ p\in\sP}}
\sum_{p^\theta<n\le p^{\theta+\eps}}\frac{\Lambda(n)(n|p)}{n\log n}.
\end{equation}
Next, let
$$
\psi_p(x):=\sum_{n\le x}\Lambda(n)(n|p)\qquad(x>0).
$$
Using the trivial bound $\psi_p(x)\le\sum_{n\le x}\Lambda(n)\ll x$ one verifies that
\begin{align*}
\sum_{p^\theta<n\le p^{\theta+\eps}}\frac{\Lambda(n)(n|p)}{n\log n}
&=\int_{p^\theta}^{p^{\theta+\eps}}\frac{d\psi_p(u)}{u\log u}
=\int_{p^\theta}^{p^{\theta+\eps}}\frac{\psi_p(u)\,du}{u^2\log u}+o(1)
\qquad(p\to\infty).
\end{align*}
Thus, after the change of variables $u\mapsto p^t$ the relation
\eqref{eq:vida2} transforms to
\begin{equation}
\label{eq:vida3}
S_1'(\theta)=
\frac{1}{2\theta}-\frac12
\lim_{\eps\to 0^+}\eps^{-1}\lim_{\substack{p\to\infty\\ p\in\sP}}
\int_{\theta}^{\theta+\eps}\frac{\psi_p(p^t)\,dt}{tp^t}.
\end{equation}
To proceed further, we use the following statement, which is a reformulation
of Lemma~\ref{lem:fame} in the special case that $\chi$ is
the Legendre symbol.

\begin{lemma}
\label{lem:techlemma4}
Let $c_4$  have the property described in Lemma~\ref{lem:fame}.
For any constant $c>0$ one has  the representation 
\begin{equation}
\label{eq:psipexpr}
\psi_p(p^t)=-\sum_{\varrho\in\sZ_p}\frac{\widetilde m(\varrho)}{\varrho}\,p^{\varrho t}
+r(p,t,c),
\end{equation}
where  the remainder term $r(p,t,c)$ admits the bound
\begin{equation}\label{bound}
r(p,t,c)=O\big(p^te^{-ct}\big)
\end{equation}
uniformly for 
\begin{equation}\label{range}
c_4^{-1}\le t\le\frac{1}{K^2} \log p,
\end{equation}
and  $\sZ_p$ denotes the set of distinct zeros
$\varrho=\beta+i\gamma$ of $L(s,(\cdot|p))$ such that
$\beta>1-2c/\log p$ and $|\gamma|\le p$, $\widetilde m(\varrho)$
is the multiplicity of any such zero, and the implied constant depends only on $c$.
Here $K=K(c)$ is the constant  described in Lemma~\ref{lem:fame}.
\end{lemma}

Before proceeding, we study the zeros set $\sZ_p$ from Lemma  \ref{lem:techlemma4}.

By  Lemma~\ref{lem:fortune} with $q=T=p$, one observes  that 
\begin{equation}\label{eq:arctic}
|\sZ_p|=\sum_{\varrho\in\sZ_p}\widetilde m(\varrho)=N(1-2c/\log p,p,(\cdot|p))\ll\exp(4cc_1),
\end{equation}
which shows that the number of zeros of the $L$-function $L(s,(\cdot|p))$ in the strip  $\{\beta +i\gamma\,|\,\beta>1-2c/\log p, \,\, |\gamma|\le p \}$ is uniformly bounded with respect to $p$. Let us define
$$
N=\limsup_{p\in \sP}|\sZ_p|.
$$
 
If $N=0$,  
the set $\sZ_p$ is empty for $p\in \sP$ large enough.
If $N>0$, then without loss of generality (i.e., replacing
$\sP$ with a suitable infinite subset of $\sP$) we may assume that 
$$
|\sZ_p|=N\qquad(p\in \sP).
$$
After the compactifying the complex plane $\mathbb{C}\to \mathbb{C}\cup \{\infty\}$, one can use  a straightforward compactness argument
to conclude, after possibly replacing $\sP$ with a suitable infinite subset of $\sP$,
that there are $N$ sequences of zeros $\{\varrho_p^{(n)}\}_{p\in\sP}$ of the $L$-functions
$L(\cdot,(\cdot | p))$ such that each sequence is contained
in the strip $\{\beta +i\gamma\,|\,\beta>1-2c/\log p, \,\, |\gamma|\le p \}$, and
\begin{equation}\label{predv}
\lim_{\substack{p\to\infty\\ p\in\sP}}(\varrho_p^{(n)}-1)\log p=\ell^{(n)}\in \mathbb{C}\cup\{\infty\}\qquad(n=1, 2, \dots , N).
\end{equation}
Here, the limits are taken with respect to the topology of the Riemann sphere  $\mathbb{C}\cup\{\infty\}$.

Denote by $\sL$  the set of \emph{distinct finite limits}
$\ell^{(n)}$ from \eqref{predv}; we always set  $\sL=\varnothing$ if $N=0$, and it can be the case that $\sL=\varnothing$ even if $N>0$.

 \begin{remark}\label{uprem}
From the definitions it is clear that
 \begin{equation}\label{uppp}
 -2c\le\Re\ell\le 0 \qquad (\ell\in \sL).
 \end{equation}
\end{remark}
  
The next lemma provides the second representation for  $S'(\theta)$.

\begin{lemma}\label{neman}
Let  $\sL$ be the zeros attractor defined above.
Then 
$$
S_1'(\theta)=
\frac{1}{2\theta}+\frac{1}{2\theta}
\sum_{\substack{\ell\in\sL\\\Re\ell>-c}}    \widehat  m(\ell)e^{\ell\theta}
+O(e^{-c\,\theta})\qquad(\theta\to \infty),
$$
where  the multiplicity  $\widehat m(\ell)$of $\ell\in \cL$ is given by
\begin{equation}\label{mult}
  \widehat  m(\ell)=\#\{ n\, |\,\ell^{(n)}=\ell\}.
\end{equation}
\end{lemma}
\begin{proof} If $N=0$ (and thus, $\sL=\varnothing$)
we claim that
\begin{equation}\label{asrep}
S_1'(\theta)=
\frac{1}{2\theta}+O(e^{-c\theta}) \qquad (\theta\to \infty).
\end{equation}
Indeed, the set $\sZ_p$ is empty   for $p\in \sP$ large enough, hence by Lemma \ref{lem:techlemma4} 
the function $\psi_p(p^t)$ admits the following uniform estimate 
\begin{equation}\label{bound0}
\psi_p(p^t)=O\big(p^te^{-ct}\big) \qquad ( p\to \infty),
\end{equation}
provided that the parameter $t$ satisfies
\begin{equation}\label{range0}
c_4^{-1}\le t\le K^{-2}\log p.
\end{equation}
Recall that by \eqref{eq:vida3} we have 
\begin{equation}\label{s2theta2}
S_1'(\theta)=
\frac{1}{2\theta}-\frac{1}{2}\lim_{\eps\to 0^+}\eps^{-1}\lim_{\substack{p\to\infty\\ p\in\sP}}
\int_{\theta}^{\theta+\eps}
 \frac{\psi_p(p^t) dt}{tp^t}\qquad (\theta>\lambda).
\end{equation}
If, in addition, $\theta>c_4^{-1}$,  then for all $t\in [\theta, \theta+\varepsilon]$ the two-sided estimate \eqref{range0} holds
provided that $\varepsilon$ is fixed and $p$ is large enough. Therefore, taking into account that the double limit  in the right hand side of \eqref{s2theta2} exists, one can use the uniform bound \eqref{bound0} to conclude that 
\begin{equation}\label{prom}
\lim_{\eps\to 0^+}\eps^{-1}\lim_{\substack{p\to\infty\\ p\in\sP}}
\int_{\theta}^{\theta+\eps} \frac{\psi_p(p^t)dt}{tp^t}=O(e^{-c\theta}) \qquad (\theta\to \infty),
\end{equation}
which together with \eqref{s2theta2} proves the representation \eqref{asrep}. This completes the proof of  the lemma in the case that $N=0$.

Next, assume that $N>0$. Using \eqref{eq:psipexpr} we have
\begin{equation}\label{s1theta}
S_1'(\theta)=
\frac{1}{2\theta}+\frac12
\lim_{\eps\to 0^+}\eps^{-1}\lim_{\substack{p\to\infty\\ p\in\sP}}
\int_{\theta}^{\theta+\eps}\bigg(
\sum_{\varrho\in\sZ_p}\frac{\widetilde m(\varrho)}{\varrho}\,p^{\varrho t}
-r(p,t,c)\bigg)
\frac{dt}{tp^t}
\end{equation}
for all $\theta>\lambda$, where
$$
r(p,t,c)=\psi_p(p^t)+\sum_{\varrho\in\sZ_p}\frac{\widetilde m(\varrho)}{\varrho}\,p^{\varrho t}.
$$
Suppose for the moment that we have shown that the limit
\begin{equation}\label{limex}
\lim_{\eps\to 0^+}\eps^{-1}\lim_{\substack{p\to\infty\\ p\in\sP}}
\int_{\theta}^{\theta+\eps}
\sum_{\varrho\in\sZ_p}\frac{\widetilde m(\varrho)}{\varrho}\,p^{\varrho t}\,\frac{dt}{tp^t}
\end{equation}
exists.  Under this assumption, \eqref{s1theta} splits as 
\begin{align*}
S_1'(\theta)&=
\frac{1}{2\theta}+\frac12
\lim_{\eps\to 0^+}\eps^{-1}\lim_{\substack{p\to\infty\\ p\in\sP}}
\int_{\theta}^{\theta+\eps}
\sum_{\varrho\in\sZ_p}\frac{\widetilde m(\varrho)}{\varrho}\,p^{\varrho t}\frac{dt}{tp^t},
\\&\qquad- \frac12
\lim_{\eps\to 0^+}\eps^{-1}\lim_{\substack{p\to\infty\\ p\in\sP}}
\int_{\theta}^{\theta+\eps}r(p,t,c)
\frac{dt}{tp^t}.
\end{align*}
Arguing as in the case $N=0$, and taking into account
the bound \eqref{bound}, one obtains that
\begin{equation}\label{prom00}
\lim_{\eps\to 0^+}\eps^{-1}\lim_{\substack{p\to\infty\\ p\in\sP}}
\int_{\theta}^{\theta+\eps} \frac{r(c,p, t)dt}{tp^t}=O(e^{-c\theta}) \qquad (\theta\to \infty);
\end{equation}
therefore,
$$
S_1'(\theta)=
\frac{1}{2\theta}+\frac12
\lim_{\eps\to 0^+}\eps^{-1}\lim_{\substack{p\to\infty\\ p\in\sP}}
\int_{\theta}^{\theta+\eps}
\sum_{\varrho\in\sZ_p}\frac{\widetilde m(\varrho)}{\varrho}\,p^{\varrho t}\frac{dt}{tp^t}+
O(e^{-c\theta}) \qquad (\theta \to \infty).
$$
Thus, to complete the proof we need to show that the double limit \eqref{limex} exists, and to verify  that 
$$
\frac12
\lim_{\eps\to 0^+}\eps^{-1}\lim_{\substack{p\to\infty\\ p\in\sP}}
\int_{\theta}^{\theta+\eps}
\sum_{\varrho\in\sZ_p}\frac{\widetilde m(\varrho)}{\varrho}\,p^{\varrho t}\frac{dt}{tp^t}=\frac{1}{2\theta}
\sum_{\ell\in\sL}   \widehat  m(\ell)e^{\ell\theta} \qquad (\theta>\lambda).
$$

For fixed $p\in\sP$ we have
\begin{align*}
\int_{\theta}^{\theta+\eps}
\sum_{\varrho\in\sZ_p}\frac{\widetilde m(\varrho)}{\varrho}\,p^{(\varrho -1)t}\frac{dt}{t}
&=\sum_{\substack{n\\\ell^{(n)}<\infty}} \frac{1}{\varrho_p^{(n)}}
\int_{\theta}^{\theta+\eps}
\,p^{(\varrho_p^{(n)} -1)t}\frac{dt}{t}
\\&\qquad+
\sum_{\substack{n\\\ell^{(n)}=\infty}} \frac{1}{\varrho_p^{(n)}}
\int_{\theta}^{\theta+\eps}
\,p^{(\varrho_p^{(n)} -1)t}\frac{dt}{t}.
\end{align*}
Using the definition  \eqref{predv} of $\ell^{(n)}$
we see that
$$
\lim_{\substack{p\to\infty\\ p\in\sP}}\sum_{\substack{n\\\ell^{(n)}<\infty}} \frac{1}{\varrho_p^{(n)}}
\int_{\theta}^{\theta+\eps}
\,p^{(\varrho_p^{(n)} -1)t}\frac{dt}{t}=\sum_{\substack{n\\\ell^{(n)}<\infty}}
 \int_{\theta}^{\theta+\eps}
\,e^{\ell^{(n)}t}\frac{dt}{t},
$$
and therefore
\begin{align*}
\frac12
\lim_{\eps\to 0^+}\eps^{-1}&\lim_{\substack{p\to\infty\\ p\in\sP}}\sum_{\substack{n\\\ell^{(n)}<\infty}} \frac{1}{\varrho_p^{(n)}}
\int_{\theta}^{\theta+\eps}
\,p^{(\varrho_p^{(n)} -1)t}\frac{dt}{t}
\\&=\frac12
\lim_{\eps\to 0^+}\eps^{-1}\sum_{\substack{n\\\ell^{(n)}<\infty}}
 \int_{\theta}^{\theta+\eps}
\,e^{\ell^{(n)}t}\frac{dt}{t}
=\frac{1}{2\theta}
\sum_{\ell\in\sL}  m(\ell)e^{\ell\theta}.
\end{align*}

Next, we show that
\begin{equation}\label{nunu}
\lim_{\substack{p\to\infty\\ p\in\sP}}\sum_{n:\, \ell^{(n)}=\infty } \frac{1}{\varrho_p^{(n)}}
\int_{\theta}^{\theta+\eps}
\,p^{(\varrho_p^{(n)} -1)t}\frac{dt}{t}=0.
\end{equation}
Indeed, if $\ell^{(n)}=\infty$, from \eqref{predv} and the estimate
\begin{equation}\label{raz0}
|1-\Re\varrho_p^{(n)}|\log p=(1-\Re\varrho_p^{(n)})\log p\le 2c,
\end{equation}
it follows that
\begin{equation}\label{dva0}
\lim_{\substack{p\to\infty\\ p\in\sP}}\lim |\omega(n,p)|=\infty,
\end{equation}
where 
$$
\omega(n,p)=\Im\varrho_p^{(n)}\log p.
$$
If $p\in \sP$ is large enough and $n$ is such that $\ell^{(n)}=\infty$,
integration by parts yields
\begin{align}
\int_{\theta}^{\theta+\eps}
\,p^{(\varrho_p^{(n)} -1)t}\frac{dt}{t}&=\int_{\theta}^{\theta+\eps}
e^{(\Re\varrho_p^{(n)} -1)\log pt}  e^{i\,\omega(n,p)t}\,\frac{dt}{t}
\nonumber \\&
=\frac{1}{i\,\omega(n,p)}\cdot e^{(\Re\varrho_p^{(n)} -1)\log pt}  e^{i\,\omega(n,p)t}\frac1t \bigg \vert_\theta^{\theta+\varepsilon}
\nonumber \\&\quad-
\frac{1}{i\,\omega(n,p)}\int_{\theta}^{\theta+\eps}  \frac{d}{dt}
\bigg(  \frac{e^{(\Re\varrho_p^{(n)} -1)\log pt}}{t}  \bigg)e^{i\,\omega(n,p)t}\,dt, \label{tri0}
\end{align}
with $\omega(n,p)\ne 0$.
Combining \eqref{raz0},  \eqref{dva0}  and \eqref{tri0} we have 
$$
\int_{\theta}^{\theta+\eps}
\,p^{(\varrho_p^{(n)} -1)t}\frac{dt}{t}=o(1) \qquad (p\to \infty, \, p\in \sP),
$$
provided that   $\ell^{(n)}=\infty$.
Taking into account the bound 
$$
\frac{1}{|\varrho_p^{(n)}|}\le \frac{1}{1-2c/\log p}<2,
$$
which holds if $p$ is large enough, \eqref{nunu} is proved. Hence, 
$$
S_1'(\theta)=
\frac{1}{2\theta}+\frac{1}{2\theta}
\sum_{\ell\in\sL}  \widehat m(\ell)e^{\ell\theta}
+O(e^{-c\,\theta})\qquad(\theta\to \infty).
$$
Finally, the summands with $\Re\ell\le -c$ can be absorbed by the error term, which  completes the proof of the lemma.
\end{proof}

\subsection{Proof of Theorem \ref{thm:b-main}}
\begin{proof}
To prove our main result, Theorem \ref{thm:b-main}, we compare two asymptotic representations for the derivative $S_1'(\theta)$ provided by 
Proposition \ref{prop:firstS1'}
$$
S_1'(\theta)=\frac{1}{2\theta}+\frac{1}{2\theta}\sum_{k\in\sK} m(k)
e^{-ik\theta}+O(e^{-c \theta})\qquad(\theta\to \infty)
$$
and  by Lemma  \ref{neman}
$$
S_1'(\theta)=
\frac{1}{2\theta}+\frac{1}{2\theta}
\sum_{\ell\in\sL:\, \Re\ell >-c}  \widehat m(\ell)e^{\ell\theta}
+O(e^{-c\,\theta})\qquad(\theta\to \infty),
$$
respectively.

Comparing those representations yields
\begin{equation}
\label{eq:snowbound3}
\sum_{\ell\in\sL:\, \Re \ell>-c}\widehat m(\ell)e^{\ell\theta}
=\sum_{k\in\sK} m(k)e^{-ik\theta}+O(\theta e^{-c\,\theta})\qquad (\theta\to \infty).
\end{equation}
However, using Lemma~\ref{lem:gift} and the fact that
$\Im\,k>-c$ for all $k\in\sK$, the resulting relation is impossible
unless it is the case that $-ik$ lies in the set $\sL\cap \{z\in \mathbb{C}\,|-c< \Im(z)\le 0\}$ and therefore in $\sL$ for every $k\in\sK$.
Since the constant $c$ can be chosen arbitrarily large, this completes the proof of Theorem~\ref{thm:b-main}.
\end{proof}
\begin{remark}\label{remrem}
We note that the upper bound \eqref{uppp} from Remark~\ref{uprem} can be improved; we have 
$$
\Re \ell<0 \qquad (\ell\in \sL).
$$
Indeed, suppose that $\Re \ell =0$ for some $\ell\in \sL$.
This means that there is   a sequence $(\varrho_p)_{p\in\sP}$
with $\varrho_p\in\sZ_p$ such that $\Re((\varrho_p-1)\log p)\to 0$.
By Lemma~\ref{lem:exceptional} it is clear that
each zero $\varrho_p$ of $L(s,(\cdot|p))$ is exceptional if $p$ is large enough;
in particular, $\varrho_p=\beta_p$ is a real simple zero.  Since 
$(\beta_p-1)\log p\to 0$, the final statement in
Lemma~\ref{lem:exceptional} implies that for all sufficiently large $p$
the set $\sZ_p$ consists only of the single zero $\beta_p$; consequently, $\ell=0$ and thus $\sL=\{0\}$.
By  Lemma \ref{neman},
$$
S_1'(\theta)=
\frac{1}{2\theta}+\frac{1}{2\theta}
+O(e^{-c\,\theta})=\frac{1}{\theta}
+O(e^{-c\,\theta})\qquad(\theta\to \infty),
$$
which is inconsistent with the asymptotics obtained in   Proposition  \ref{prop:firstS1'}.

A  similar reasoning shows that the set $\sL$ is also free of (negative) reals,
that is, 
$$
\sL\cap (-\infty, 0)=\varnothing.
$$

Indeed,  since  $$
 \widehat  \distr(-i\beta)=\int_\kappa^\lambda e^{\beta x}\distr (x)dx>\int_\kappa^\lambda \distr (x)dx=1\quad (\beta>0),
 $$ 
one concludes that the equation $\widehat  \distr(k)=1$ has no zeros on the negative imaginary axis. Since the roots of  the equation  $\widehat  \distr(k)=1$ are complex conjugates of each other, using  
Proposition \ref{prop:firstS1'} one observes that the higher order terms  
in the asymptotic expansion  for the derivative $S_1'(\theta)$,$$
S_1'(\theta)=\frac{1}{2\theta}+\frac{1}{2\theta}\sum_{k\in\sK} m(k)\cos( \Re (k) \theta )
e^{-|\Im (k)|\theta}+O(e^{-c \theta})\qquad(\theta\to \infty),
$$
 are oscillatory in any order, which is inconsistent with the asymptotics provided by   Lemma \ref{neman}
$$
S_1'(\theta)=
\frac{1}{2\theta}+\frac{1}{2\theta}
\sum_{\substack{\ell\in\sL\\\Re\ell>-c}}    \widehat  m(\ell)e^{\ell\theta}
+O(e^{-c\,\theta})\qquad(\theta\to \infty),
$$
unless $\Im(\ell)\ne 0 $ for all $\ell\in \sL$.

In particular, the sequence of the $L$-functions $\{L(s,(\cdot|p))\}_{p\in \sP}$ is allowed to have  only finitely many terms that have an  exceptional zero  $\varrho_p$
satisfying the bound
$$
\varrho_p>1-\frac{c_2}{\log p}
$$
 from Lemma \ref{lem:exceptional}.

\end{remark}

\section{Connection with Heath-Brown's result}
\label{sec:heath-brown}

The main goal of this section is to show how to deduce the Heath-Brown result concerning the behavior of the Dirichlet $L$-function  from our more general considerations.

We start, however, with preliminary considerations  where we discuss the optimality of our main assumptions concerning the distribution of quadratic nonresidues, namely the hypotheses
\eqref{eq:n0pkap},
\eqref{eq:Np-est}
and \eqref{eq:intrho}. 

\begin{theorem}\label{low} Assume the $(\kappa,\lambda)$ hypotheses of Theorem~\ref{thm:b-main},
i.e., that the conditions of \eqref{eq:n0pkap}, \eqref{eq:Np-est} and \eqref{eq:intrho} are met. 
Then the following inequality 
\begin{equation}
\label{eq:kappa_bound}
\kappa\le \lambda/\sqrt{e}
\end{equation}
necessarily holds. Moreover, in the case that 
\begin{equation}\label{eq:kappa_bound1}
\kappa=\lambda/\sqrt{e}, 
\end{equation}
the density $\distr$ is given by
\begin{equation}\label{form1}
\distr(\theta)=\frac2\theta \,X_{\kappa,\lambda}(\theta),
\end{equation}
where $X_{\kappa,\lambda}(\cdot)$ is the indicator function of the interval $[\kappa, \lambda]$.
Finally, there is an absolute constant $\lambda_0>0$ such that if
$\lambda<\lambda_0$, then the strict inequality 
 \begin{equation}\label{eq:kappa_bound2}
\kappa<\lambda/\sqrt{e} 
\end{equation}
necessarily holds.
 
\end{theorem}
\begin{proof} {\it Step 1}. First, we show that \eqref{eq:kappa_bound} holds.
Indeed, suppose on the contrary that 
\begin{equation}\label{nea}
\kappa>\lambda/\sqrt{e}.
\end{equation}
 By 
Proposition \ref{prop:S1-converted} we have the representation 
$$ S_1(\theta)=\sum_{n\in\N}
\frac{(\delta*\distr^{*(n-1)})(\theta)}{n}\qquad(\theta\ge 0).
$$
Since both $\delta$ and $\distr$ vanish on $[0,\kappa)$ by hypothesis,  the convolution  $\delta*\distr^{*(n-1)}$ vanishes on the interval $[0, n\kappa)$ and therefore
\begin{equation}\label{s=l}
 S_1(\theta)=\delta(\theta) \qquad(0\le \theta<2\kappa).
\end{equation}
From \eqref{nea} it follows that $\lambda<2\kappa$
and thus
$$
S_1(\lambda)=\delta(\lambda)=1/2.
$$
Here, the last equality follows from our hypothesis that the distribution $\distr$ is supported on
$[\kappa,\lambda]$ and that
$$
\delta(\lambda)=\frac12\int_0^\lambda\distr(u)\,du=\frac12.
$$
Recalling that  
$$S_{1}(\theta)=\lim_{\substack{p\to\infty\\ p\in\sP}}
\sum_{\substack{q\le
p^\theta\\
(q|p)=-1}}\frac{1}{q}=\lim_{\substack{p\to\infty\\ p\in\sP}}
\sum_{\substack{p^\kappa<q\le
p^\theta\\
(q|p)=-1}}\frac{1}{q}
$$
and using Mertens' theorem, one gets the bound
$$
S_{1}(\lambda)= \lim_{\substack{p\to\infty\\ p\in\sP}}
\sum_{\substack{p^\kappa<q\le
p^\lambda\\
(q|p)=-1}}\frac{1}{q}\le \lim_{p\to \infty}
\sum_{p^\kappa<q\le
p^\lambda}\frac{1}{q}=\log\frac\lambda\kappa.
$$
Hence
$$
\frac12\le \log\frac\lambda\kappa,
$$
which is inconsistent with \eqref{nea}. The proof of \eqref{eq:kappa_bound} is complete.

\newpage
{\it Step 2}.
Next,  we show that 
\eqref{eq:kappa_bound1} implies  \eqref{form1}.
Using   Mertens' theorem again we    obtain that 
\begin{align*}
 \frac12&=S_1(\lambda)=S_1(\theta)+(S_1(\lambda)-S_1(\theta))\\
 &\le S_1(\theta)+\log \frac\lambda\theta\le\log\frac\theta\kappa+
 \log \frac\lambda\theta=\log\frac\lambda\kappa=\frac12,
\end{align*}
 and thus
 \begin{equation}\label{ex}
 S_1(\theta)=\log\frac\theta\kappa \qquad (\kappa\le \theta\le\lambda).
 \end{equation}
It remains to observe that  $\lambda=\kappa\sqrt{e}<2\kappa$,
and therefore one can use  \eqref{s=l} to conclude that 
 $$
 \delta(\theta)= S_1(\theta)=\log\frac\theta\kappa  \qquad(\kappa\le \theta\le\lambda),
 $$ 
 which proves \eqref{form1} in view of the equality
 \begin{equation}
 \label{eq:intrho1}
\delta(\theta)=\frac12\int_0^\theta\distr(u)\,du\qquad(\theta\ge 0).
\end{equation}
 
{\it Step 3}. Finally, we  prove  the remaining assertion of the theorem for
$$
\lambda_0=\frac{c_2}{4|\Im(k_0)|},
$$
where 
 $c_2$  is the constant from Lemma~\ref{lem:exceptional},
and $k_0$ is one of the roots of the   equation
\begin{equation}
\label{nini}
2\int_{1/(4\sqrt{e})}^{1/4} e^{i kx}\,\frac{dx}{x}=1.
\end{equation}
that lie closest to the real axis.

Suppose that  inequality \eqref{eq:kappa_bound2} does not hold;
then, by Step 1, we have  
 \begin{equation}\label{jam}
\kappa=\lambda/\sqrt{e}.
\end{equation}
By Step 2, the probability distribution $\distr$ is given by
$$
\distr(\theta)=\distr_\lambda(\theta)=\frac2\theta X_{\kappa,\lambda}(\theta).
$$
Evaluating the Fourier transform of the function $d_\lambda$ 
 we have
\begin{align*}
\widehat d_\lambda(k):&=2\int_{\lambda/\sqrt{e}}^\lambda e^{ikx}\,\frac{dx}{x}
=2\int_{1/(4\sqrt{e})}^{1/4} e^{4i\lambda kx}\,\frac{dx}{x}.
\end{align*}
Thus, the zeros of the equation 
$$
\widehat d_\lambda(k)=1
$$
can obtained from the roots of the equation 
\eqref{nini}
by rescaling $k\to 4\lambda k$.

 By Theorem 1.2, there is a complex sequence $\varrho_p$ with $L(\varrho_p, (\cdot |p))=0 $ such that 
$$
(\varrho_p-1) \log p \to -ik_0.
$$
Since, by hypothesis, 
$$
\lambda<\lambda_0=\frac{c_2}{4|\Im(k_0)|},
$$
one obtains that 
$$
1-\Re(\varrho_p)< 4\lambda \frac{|\Im(k_0)|}{\log p}< 4\lambda_0 \frac{|\Im(k_0)|}{\log p}=\frac{c_2}{\log p}
 \qquad (p \text{ is large enough}).
$$
By Lemma 4.2 the roots $\varrho_p$ are exceptional zeros if $p$ is large enough, and hence 
$$
\sL\cap (-\infty, 0)\ne\varnothing,
$$
where $\sL$ the is zeros attractor defined by \eqref{predv}.
However, this is impossible in view of Remark~\ref{remrem},
and this contradiction completes the proof.
\end{proof}

\begin{remark}
It follows from Theorem \ref{low} that  the inequality \eqref{eq:kappa_bound} cannot be relaxed.
Thus we see that the hypotheses \eqref{eq:n0pkap}, \eqref{eq:Np-est}
and \eqref{eq:intrho} together imply that the probability distribution
$\distr$ cannot be ``too concentrated" in a certain sense;
in particular, the convex hull of the support of the probability
distribution must always contain the critical interval
$\big[\lambda/\sqrt{e},\lambda\big]$.
\end{remark}

In the case that $\lambda=1/4$, the  exponent
$$
\kappa=\frac{1}{4\sqrt{e}}
$$
coincides with the exponent in the Burgess bound.
 In particular, from Theorem \ref{low}  it follows that if $\lambda=1/4$ and $\kappa=1/(4\sqrt{e})$, then
 $$
 \delta(\theta)=\log 4 \theta \sqrt{e} \qquad\big(1/(4\sqrt{e})\le  \theta\le 1/4\big),
 $$
 provided that  the requirements \eqref{eq:Np-est} and \eqref{eq:intrho} are met. 

As it turns out, in the special case that $\kappa=\lambda/\sqrt{e}$,
one can replace the hypotheses \eqref{eq:Np-est} and \eqref{eq:intrho}
 by a much weaker condition (see \eqref{ee} below) and obtain a considerably stronger result,  which automatically guarantees the existence
of the probability distribution $\distr$ of the form \eqref{form1}.

\begin{lemma}\label{vagno} Suppose that  $0<\lambda\le 1/4$,
and let   $\kappa=\lambda/\sqrt{e}$. Assume  the hypothesis
\eqref{eq:n0pkap}. Suppose, in addition, that
\begin{equation}
\label{ee}
\lim_{\substack{p\to\infty\\ p\in\sP}}\frac{\big |\sN_p(p^\theta)\big|}{p^\theta}=\frac12\qquad (\theta\ge \lambda).
\end{equation}
Then  for all $\theta\ge0$ one has
\begin{equation}
\label{eq:Np-estee}
\lim_{\substack{p\to\infty\\ p\in\sP}}\frac{\big |\sN_p(p^\theta)\big|}{p^\theta}=\delta(\theta),
\end{equation}
where the density function $\delta$ has the form
\begin{equation}
\label{eq:intrhoee}
\delta(\theta):=\frac12\int_0^\theta\distr(u)\,du\qquad(\theta\ge 0)
\end{equation}
with 
$$
\distr(\theta)=\frac2\theta \qquad (\kappa\le\theta\le\lambda).
$$ 

\end{lemma}
\begin{proof}
If $0\le\theta\le \kappa$, then \eqref{eq:Np-estee} is trivial with $\delta(\theta)=0$ 
in view of \eqref{eq:n0pkap}, whereas if $\theta\ge \lambda$,
then \eqref{eq:Np-estee} holds  with $\delta(\theta)=1/2$ by hypothesis.

Now suppose that $\kappa=\lambda/\sqrt{e}<\theta<\lambda$, and let $p\in\sP$ be fixed.
Since $2 \kappa>\theta$ it is
clear that a natural number $n\le p^\theta$ is a nonresidue if
and only if $n=qm$ for some prime nonresidue~$q$ and natural number $m$,
and in this case the pair $(q,m)$ is determined
uniquely by~$n$. Therefore,
\begin{equation}
\label{eq:expr1ee}
\big|\sN_p(p^\theta)\big|=\sum_{\substack{q\le p^\theta\\
(q|p)=-1}}\fl{\frac{p^\theta}{q}}=
S_{p,1}(\theta)\,p^\theta+O\bigg(\frac{p^\theta}{\log p}\bigg),
\end{equation}
where (as before)
$$
S_{p,1}(\theta):=\sum_{\substack{q\le p^\theta\\
(q|p)=-1}}q^{-1}.
$$
Note that we have used the Prime Number Theorem and the fact that
$\theta\gg 1$ to bound the error term in~\eqref{eq:expr1ee}.

From \eqref{ee} and \eqref{eq:expr1ee} one concludes that  the limit 
$$
\frac12=S_1(\lambda)=\lim_{\substack{p\to\infty\\ p\in\sP}}S_{1,p}(\lambda)
$$ 
exists. Using Mertens'  theorem, we have
\begin{align*}
 \frac12&=S_1(\lambda)=
 \liminf_{\substack{p\to\infty\\ p\in\sP}}S_{1,p}(\lambda)+\limsup_{\substack{p\to\infty\\ p\in\sP}}(S_1(\lambda)-S_{1,p}(\theta))\\
&\le \limsup_{\substack{p\to\infty\\ p\in\sP}}S_{1,p}(\theta)+\log \frac\lambda\theta\le\log\frac\theta\kappa+
 \log \frac\lambda\theta=\log\frac\lambda\kappa=\frac12,
\end{align*}
 and hence
\begin{align*}
 S_1(\theta)&=\liminf_{\substack{p\to\infty\\ p\in\sP}}
 S_{1,p}(\theta)
 =\limsup_{\substack{p\to\infty\\ p\in\sP}}
 S_{1,p}(\theta)
=\lim_{\substack{p\to\infty\\ p\in\sP}}
 S_{1,p}(\theta)
 =\log\frac\theta\kappa
\end{align*}
for all $\theta\in(\kappa,\lambda)$.
In turn, by  \eqref{eq:expr1ee} one gets that 
$$\lim_{\substack{p\to\infty\\ p\in\sP}}\frac{\big |\sN_p(p^\theta)\big|}{p^\theta}=\delta(\theta) \qquad (\kappa< \theta<\lambda),
$$
thus \eqref{eq:Np-estee} holds in the full range of $\theta\ge 0$.
\end{proof}
 
As  a corollary of Theorem~\ref{thm:b-main},
we are in a position to give an independent proof  the following result,
which is originally due to Heath-Brown
(for a more precise statement, see Diamond \emph{et al}~\cite[Appendix]{DMV}, where
a reconstruction of Heath-Brown's work is given).

\begin{corollary} Suppose that 
\begin{equation}
\label{eq:ironmanee}
(n|p)=1\qquad(1\le n\le p^{1/(4\sqrt{e})})
\end{equation}
for all primes $p$ in some infinite set $\sP$. Then for every zero
$z$ of the function
$$
H(z):=\frac{2}{z}\int_{1/\sqrt{e}}^1(1-e^{-zu})\,\frac{du}{u},
$$
there is a sequence $(\varrho_p)_{p\in\sP}$ such that each term $\varrho_p$
is a zero of the $L$-function $L(s,(\cdot|p))$, and
\begin{equation}
 \label{hb}(1-\varrho_p)\log p=-4z+\error.
 \end{equation}

\end{corollary}
\begin{proof}
Combining  \eqref{eq:ironmanee} with Lemma~\ref{lem:hildebrand}, we see that
the hypotheses of Lemma \ref{vagno} are met with 
$$
\kappa=\frac{1}{4\sqrt{e}}\qquad \text{and}\qquad  \lambda=\frac14.
$$
Therefore, the hypotheses
\eqref{eq:n0pkap},
\eqref{eq:Np-est}
and \eqref{eq:intrho} are met with 
 the function $\delta$ and the probability distribution $\distr$ given by \eqref{eq:dragon}
and \eqref{eq:Burgessrho}, respectively.
For the probability distribution $\distr$ given by \eqref{eq:Burgessrho}
one easily verifies that
$$
H\Big(-\frac{ik}{4}\Big)=\frac{4i}{k}\big(1-\widehat\distr(k)\big);
$$
therefore, the asymptotic representation  \eqref{hb}
follows from Theorem \ref{thm:b-main}.
\end{proof}

\section*{Acknowledgements}

The authors thank Ahmet G\"ulo\u glu, Alex Koldobsky,
Mikhail Lifshits, Wesley Nevans,
Mark Rudelson and Bob Vaughan for helpful conversations.

\end{document}